\documentclass[reqno,a4paper]{amsart}
\usepackage{eucal,amsfonts,amssymb,amsmath,amsthm,epsfig,mathrsfs}
\usepackage{xcolor}

\usepackage{amscd,amsxtra}
\usepackage{enumerate}
\usepackage{latexsym}
\usepackage{amsfonts}
\usepackage{amsmath}
\usepackage{amssymb}
\usepackage{amsthm}
\usepackage{graphicx}
\allowdisplaybreaks

 \makeatletter \@addtoreset{equation}{section}

\makeatother \makeatletter

\newtheorem{theorem}{Theorem}[section]
\newtheorem{hyp}[theorem]{Hypothesis}{\rm}
\newtheorem{lemma}[theorem]{Lemma}
\newtheorem{corollary}[theorem]{Corollary}

\newtheorem{remark}[theorem]{Remark}

\newcommand{\C}{{\mathbb C}}
\newcommand{\R}{{\mathbb R}}
\newcommand{\N}{{\mathbb N}}

\newcommand{\no}{\nonumber}
\newcommand{\ve}{\varepsilon}
\renewcommand{\l}{\lambda}
\newcommand{\s}{\sigma}
\newcommand{\om}{\omega}
\newcommand{\Om}{\Omega}
\renewcommand{\a}{\alpha}
\renewcommand{\b}{\beta}
\newcommand{\ov}{\overline}
\newcommand{\qq}{\qquad}
\newcommand{\q}{\quad}
\newcommand{\rmd}{{\rm d}}

\title[A strongly ill-posed problem in an unbounded domain]{A strongly ill-posed problem
for a degenerate parabolic equation with unbounded coefficients in an unbounded domain $\Omega\times {\mathcal O}$ of $\R^{M+N}$}
\author{A. Lorenzi}
\address{Dipartimento di Matematica, Universit\`a degli Studi di Milano, Via Saldini 50, I-20133 MILANO (Italy)}
\email{alfredo.lorenzi@unimi.it}
\author{L. Lorenzi}
\address{Dipartimento di Matematica e Informatica, Universit\`a degli Studi di Parma, Parco Area delle Scienze 53/A, I-43124 PARMA (Italy)}
\email{luca.lorenzi@unipr.it}

\begin{document}

\begin{abstract}
In this paper we deal with a strongly ill-posed second-order degenerate parabolic problem
in the unbounded open set $\Om\times {\mathcal O}\subset \R^{M+N}$, related to a linear equation with unbounded coefficients, with no initial condition,
but endowed with the usual Dirichlet condition on $(0,T)\times \partial(\Om\times {\mathcal O})$ and an additional condition involving the $x$-normal derivative on $\Gamma\times {\mathcal O}$, $\Gamma$ being an open subset of $\Om$.

The task of this paper is twofold: determining sufficient conditions on our data implying the uniqueness of the solution $u$ to the boundary value problem as well as determining a pair of metrics with respect of which $u$ depends continuously on the data.

The results obtained for the parabolic problem are then applied to a similar problem for a convolution integrodifferential linear parabolic equation.
\end{abstract}
\maketitle

\section{Introduction}
In the second  half of the last century a lot of interest, due to the rushing on of Technology, was devoted to Inverse Problems, a branch of which consists just of strongly ill-posed problems, where {\it strongly} means that no transformation can be found in order to change such problems to well-posed ones, at least, say, when working in classical or Sobolev function spaces of {\it finite order}.

Assume that you are dealing with the evolution of the temperature $u$ involving a body $\om$ occupying a (possibly) unbounded domain in $\R^{M+N}$, and assume that you cannot measure the temperature $u$ inside $\om$, but you can perform only measurements on the boundary of $\om$. So, you have no initial condition at your disposal, but only several boundary measurements of temperature, flux and so on. This makes the parabolic problem strongly ill-posed. The basic questions which arise in this case are the following:
 \begin{enumerate}[\rm (i)]
\item
may the solution to this problem be unique?
\item
in this case may the solution depend continuously on the boundary data?
\item
if this is possible, which are the allowed metrics?
\end{enumerate}
This paper is devoted to shed some light on degenerate parabolic problems of that kind on (possibly) unbounded
domains $\om=\Om\times {\mathcal O}$, where $\Om\subset \R^M$ and ${\mathcal O}\subset \R^N$ are two smooth open sets, the
first being bounded, while the latter is unbounded.
More precisely, we consider operators ${\mathcal A}$, defined on smooth functions $\zeta:\Omega\times {\mathcal O}\to\R$ by
\begin{align*}
{\mathcal A}\zeta(x,y)=&
{\rm div}_x(a(x)\nabla_x\zeta(x,y))
+\sum_{i=1}^Mc_i(x,y)D_{x_i}\zeta(x,y)\\
&+ \sum_{j=1}^Nb_j(y)D_{y_j}\zeta(x,y)+b_0(x,y)\zeta(x,y),
\end{align*}
for any $(x,y)\in\Omega\times {\mathcal O}$. We assume that the function $a$ nowhere vanishes in $\overline\Omega$. Anyway operator ${\mathcal A}$ is
degenerate since its leading part contains second-order derivatives computed only with respect to the variables $x_1,\ldots,x_N$.

We will be concerned mainly with the questions of uniqueness and continuous dependence on the data (two fundamental topics for people working in Applied Mathematics) of the nonhomogeneous linear parabolic equation associated with the operator  ${\mathcal A}$ in $(0,T)\times\Omega\times{\mathcal O}$, with no initial conditions. The lack of the initial conditions is replaced by the requirement that the ``temperature'' $u$ should assume prescribed values on $(0,T)\times \partial (\Om\times {\mathcal O})$, while the $x$-normal derivative of $u$ should assume prescribed values on an open subsurface $(0,T)\times \Gamma\times {\mathcal O}$ of the lateral boundary $(0,T)\times \partial \Om\times {\mathcal O}$.


The fundamental tool to give some positive answer to our problem are new Carleman estimates that fits our case.
Following the ideas in \cite[Theorem 3.4]{CH2}, we will construct suitable Carleman inequalities related to an unbounded open set.

We then show that our technique can be adapted to deal also with some degenerate integrodifferential parabolic boundary value problems and with
some class of degenerate semilinear boundary value problems.


Carleman estimates, entering many applications in Control theory (see e.g., \cite{FI,T}) and in unique continuation theorems (see e.g., \cite{Kenig}) have shown to be a powerful tool in studying inverse and ill-posed problems for partial differential equations. Starting from the pioneering works in the eighties by Bukhgeim and Klibanov (see \cite{B2,K1,K2} and also the monographs \cite{B1,KT} and
the survey papers \cite{I,K3}), Carleman estimates have been used to solve identification problems, mainly in bounded domains, associated with nondegenerate differential operators.
We quote, e.g., \cite{BMO,B,BGLR,BY,CGR,IY1,PY,W}. On the contrary, to the best of our knowledge, Carleman estimates have not been extensively used so far
 in the analysis of inverse problems in unbounded domains. We are aware only of the papers \cite{CCG,CG}. In \cite{CCG} Carleman estimates have been used to uniquely recover the unknown function $c$ in a Cauchy problem for the Schr\"odinger equation
\begin{eqnarray*}
i\frac{\partial q}{\partial t}+{\rm div}(c\nabla q)=0,
\end{eqnarray*}
related to a strip of $\R^2$, from the knowledge of the normal derivative of the time derivative of $q$ on
on the upper boundary of the strip. More recently, in \cite{CG} the authors have
considered the more general form of the Schr\"odinger equation
\begin{eqnarray*}
i\frac{\partial q}{\partial t}+a\Delta q+bq=0,
\end{eqnarray*}
and they have shown that the
knowledge of the normal derivative of the second-order time derivative of the solution on the same part of the boundary of
the strip as in \cite{CCG}, allows for recovering the two functions $a$ and $b$.
Both in \cite{CCG} and in \cite{CG} a nondegeneracy condition is assumed on the elliptic part of the operator. Moreover, the coefficients
are assumed to be at least bounded.

Similarly, Carleman estimates for degenerate parabolic problems seem to have not been so far widely used to solve inverse problems.
We are aware of the papers \cite{CTY,T1,V,V1}.
In \cite{CTY,T1} Carleman estimates are used to recover the unknown function $g$ entering the degenerate one-dimensional heat equation
\begin{eqnarray*}
\frac{\partial u}{\partial t}-\frac{\partial}{\partial x}\left (x^{\alpha}\frac{\partial u}{\partial x}\right )=g,
\end{eqnarray*}
related to the spatial domain $(0,1)$, and where $\alpha\in [0,2)$.

In \cite{V,V1} such estimates are used to solve an identification problem for a boundary value problem associated with the heat equation
\begin{eqnarray*}
\frac{\partial u}{\partial t}=\Delta u+\frac{\mu}{|x|^2}u+g
\end{eqnarray*}
in a boundary open set containing $0$, with no initial condition and $\mu$ is a positive constant not larger than the optimal constant in Hardy's inequality. The Carleman estimates obtained by the author extends similar estimates obtained in \cite{E,VZ}.

On the other hand, Carleman estimates for degenerate parabolic equations have been more widely used in Control Theory, but mainly associated to one-dimensional parabolic operators (we quote e.g., \cite{ACF,CFR,CMV0,CMV1,CMV,CMV2,FdT} and the reference therein).

At present, we are not aware of other papers where Carleman estimates are proved for \emph{degenerate} parabolic operators with \emph{unbounded} coefficients, which are related to an \emph{unbounded} spatial domain.

The plan of the paper is the following: in Section \ref{sect-2} we exactly state the ill-posed degenerate differential problem,
while in Section \ref{sect-3} we prove two theorems involving Carleman estimates for our problem, implying the uniqueness of our solution. Section \ref{sect-4} is devoted to finding a continuous dependence result for the solution to our problem in the usual space $L^2(\Om\times {\mathcal O})$.
Finally, in Section \ref{sect-5} we extend our results to both to a convolution integrodifferential equation (see Subsection \ref{subsect-5.1}) and
to a class of semilinear equations (see Subsection \ref{subsect-5.2}).

\subsection*{Notations}
Throughout the paper we denote by $\|f\|_{\infty}$ the sup-norm of a given bounded function $f$. If $f\in C^k(\overline\Omega)$ for some $k\in\mathbb N$ and some bounded domain $\Omega\subset\mathbb R^M$, we denote by
$\|f\|_{k,\infty}$ the Euclidean norm of $f$, i.e., $\|f\|_{k,\infty}=\sum_{|\alpha|\le k}\|D^{\alpha}f\|_{\infty}$. The same notation is used to denote the $W^{k,\infty}$-norm of a function $a$.

Given a square matrix $B$, we denote by $\|B\|$ its Euclidean norm.

Typically, the function spaces that we consider consist of real-valued functions but in Sections \ref{sect-2} and \ref{sect-4}, and in the first part of Section \ref{sect-3}, where we need complex-valued functions for our integrodifferential application. In this case we use the subscript ``$\C$'' to denote function spaces consisting of complex-valued functions.

The inner product in $\R^K$ will be denoted by ``$\cdot$''.
The $L^2$-Euclidean norm, and the associated scalar product are denoted, respectively, by $\|\cdot\|_2$ and $(\cdot,\cdot)_2$.

\section{Statement of the ill-posed problem concerning a degenerate parabolic operator in $\Omega\times{\mathcal O}$}
\label{sect-2}
\setcounter{equation}{0}
Let $\Omega\subset \R^M$ and ${\mathcal O}\subset \R^N$ be two open sets of classes $C^3$ and $C^2$, respectively, the first being bounded, the latter unbounded. In particular, also ${\mathcal O}=\R^N$ is allowed.

For any fixed $T>0$, we consider the following problem: {\it look for a function $u\in H^1((0,T);L^2_{\C}(\Om\times {\mathcal O}))\cap L^2((0,T);{\mathcal H}^2_{\C}(\Om\times {\mathcal O}))$ satisfying the following boundary value problem:}
\begin{equation}
\left\{
\begin{array}{l}
D_tu(t,x,y)={\rm div}_x(a(x)\nabla_xu(t,x,y))
+c(x,y)\cdot\nabla_xu(t,x,y)+b_0(x,y)u(t,x,y)\\[1mm]
\qquad\qquad\qquad\, + b(y)\cdot\nabla_yu(t,x,y) + g(t,x,y),\\[1mm]
\qquad\qquad\qquad\qquad\qquad\qquad\qquad\qquad\qquad\;\;\; (t,x,y)\in [0,T]\times\Omega\times{\mathcal O}=:Q_T,
\\[1mm]
u(t,x,y)=h(t,x,y),\qquad\qquad\qquad\qquad\quad\;\;\, (t,x,y)\in [0,T]\times\partial_*(\Omega\times{\mathcal O}),
\\[1mm]
D_{\nu}u(t,x,y)=D_{\nu}h(t,x,y),\qquad\qquad\qquad\quad\;\! (t,x,y)\in [0,T]\times\Gamma\times{\mathcal O}.
\end{array}
\right.
\label{5.1}
\end{equation}
Here, $\partial_*(\Omega\times{\mathcal O})=(\partial\Omega\times{\mathcal O})\cup (\Omega\times\partial{\mathcal O})$,
$\Gamma$ is an open subset of $\partial\Omega$,  $\nu=\nu(x)$ denotes the outward normal unit-vector at $x\in \Gamma$ and
\begin{eqnarray*}
{\mathcal H}^2_{\C}(\Om\times {\mathcal O})=\{z\in H^2_{\C}(\Om\times {\mathcal O}): |b||\nabla_yz|\in L^2(\Om\times {\mathcal O})\}.
\end{eqnarray*}

The hypotheses on the coefficients $a$, $b_0$, $b$, $c$ and the data $g$ and $h$ are listed here below.
\begin{hyp}
\label{hyp-1}
The following conditions are satisfied.
\begin{enumerate}[\rm (i)]
\item
$a\in W^{2,\infty}(\Om)$ and there exists a positive constant $a_0$ such that $|a(x)|\ge a_0$ for any $x\in\Omega$;
\item
$b_0\in L^{\infty}_{\C}(\Om\times{\mathcal O})$;
\item
$b=(b_1,\ldots,b_N)\in (W^{1,\infty}_{\rm loc}({\mathcal O}))^N$ and ${\rm div}\,b\in L^{\infty}({\mathcal O})$;
\item
$c=(c_1,\ldots,c_M)\in (L^{\infty}_{\C}(\Om\times{\mathcal O}))^M$;
\item
$g\in L^2_{\C}(Q_T)$;
\item
$h\in H^1((0,T);L^2_{\C}(\Om\times {\mathcal O}))\cap L^2((0,T);{\mathcal H}^2_{\C}(\Om\times {\mathcal O}))$.
\end{enumerate}
\end{hyp}

Performing the translation $v=u-h$, we can change our problem to one with vanishing boundary value data: {\it look for a function
$v\in H^1((0,T);L^2_{\C}(\Om\times {\mathcal O}))\cap L^2((0,T);{\mathcal H}^2_{\C}(\Om\times {\mathcal O}))$
satisfying the boundary value problem}
\begin{equation}
\label{5.4}
\left\{
\begin{array}{l}
D_tv(t,x,y)={\rm div}_x(a(x)\nabla_xv(t,x,y))+ c(x,y)\cdot\nabla_xv(t,x,y)
\\[1mm]
\qquad\qquad\quad\quad\, +b(y)\cdot\nabla_yv(t,x,y) +b_0(x,y)v(t,x,y)+ {\widetilde g}(t,x,y),\\[1mm]
\qquad\qquad\qquad\qquad\qquad\qquad\qquad\qquad\qquad \qquad\; (t,x,y)\in Q_T,\\[1mm]
v(t,x,y)=0, \qquad\qquad\qquad\qquad\qquad\qquad\qquad\;\, (t,x,y)\in [0,T]\times\partial_*(\Omega\times{\mathcal O}),\\[1mm]
D_{\nu}v(t,x,y)=0,\qquad\qquad\qquad\qquad\qquad\qquad\;\;\;\,\, (t,x,y)\in [0,T]\times\Gamma\times{\mathcal O},
\end{array}
\right.
\end{equation}
where
\begin{align}
{\widetilde g}(t,x,y)=& g(t,x,y)-D_th(t,x,y)+{\rm div}_x(a(x)\nabla_xh(t,x,y))+ c(x,y)\cdot\nabla_xh(t,x,y)
\no\\[1mm]
&+b(y)\cdot\nabla_yh(t,x,y)+b_0(x,y)h(t,x,y),
\label{5.5}
\end{align}
for any $(t,x,y)\in (0,T)\times\Omega\times {\mathcal O}$.

\begin{remark}
\label{rem-1}
{\rm If function $a$ is {\it strictly positive}, then the differential equation in (\ref{5.4}) is {\it forward degenerate parabolic}, while if function $a$ is {\it strictly negative}, then the differential equation in (\ref{5.4}) is {\it backward parabolic}. However, in the present case this is not a trouble at all. Indeed, introducing the new unknown
\begin{eqnarray*}
{\widetilde v}(t,x,y)=v(-t,x,y),\qquad\;\, (t,x,y)\in [0,T]\times\Omega\times{\mathcal O},
\end{eqnarray*}
it is immediate to check that ${\widetilde v}$ solves the problem (\ref{5.4}) with $a(x)$, $b(y)$, $c(x,y)$, $b_0(x,y)$, ${\widetilde g}(t,x,y)$ being replaced, respectively, by $-a(x)$, $-b(y)$, $-c(x,y)$, $-b_0(x,y)$, $-{\widetilde g}(-t,x,y)$, i.e., $\widetilde v$ solves a problem with a differential {\it forward degenerate parabolic} equation.}
\end{remark}

\section{Carleman estimates for the ill-posed problem (\ref{5.1})}
\label{sect-3}
\setcounter{equation}{0}
In view of Remark \ref{rem-1}, in this section we assume that function $a$ is {\it strictly positive}, i.e.,
$a(x)\ge a_0>0$ for any $x\in {\ov \Om}$.

In order to obtain a Carleman estimate related to the domain $\Om\times {\mathcal O}$ we need a weight function, defined on $\ov{\Om}$, with special properties. The existence of such a function is proved by extending \cite[Lemma 1.1]{FI}  to the $C^3$-case and then \cite[Lemma 2.3]{IY}. This can be done without a great efforts. For this reasons, the details are left to the reader.
\begin{lemma}
\label{lemma-1}
There exists a function $\psi\in C^3(\ov{\Om})$ with the following properties:
\begin{enumerate}[\rm (i)]
\item
$\psi$ is positive in $\Omega$;
\item
$|\nabla_x \psi (x)|>0$ for any $x\in {\ov {\Om}}$;
\item
$D_{\nu}\psi(x)\le 0$ for any $x\in \partial \Om\setminus \Gamma$.
\end{enumerate}
\end{lemma}

\medskip

For any $\rho\ge 1$ we set
\begin{equation}
\varphi_\rho(t,x)= \frac{{\rm e}^{\rho \psi (x)}-{\rm e}^{2\rho \|\psi\|_{\infty}}}{t(T-t)},\qquad\;\, (t,x)\in (0,T)\times\overline\Omega,
\label{form-phi}
\end{equation}
and, for simplicity, in the rest of this section we set $\ell(t)=t(T-t)$.

In the following lemma we list some crucial estimate of the function $\varphi_{\rho}$ that we need in the proof
of the Carleman estimates.

\begin{lemma}
\label{stima-varphi-rho}
For any $x\in\overline\Omega$,
$$
\lim_{t\to 0^+}\varphi_{\rho}(t,x)=\lim_{t\to T^-}\varphi_{\rho}(t,x)=+\infty.
$$
Further, let $\alpha$ denote the positive infimum of the function $|\nabla_x\psi|$. Then, the following pointwise inequalities hold true:
\begin{align}
&\ell^{-1}\le \frac{1}{\alpha\rho}|\nabla_x\varphi_{\rho}|,
\label{lemma-stima-1}
\\[2mm]
&|D_t\varphi_\rho|\le \frac{T}{\alpha^2}{\rm e}^{2\rho \|\psi\|_{\infty}}|\nabla_x\varphi_\rho|^2
\le \frac{T^3}{4\alpha^3}{\rm e}^{2\rho \|\psi\|_{\infty}}|\nabla_x\varphi_\rho|^3,
\label{lemma-stima-2}
\\[2mm]
&|D_t^2\varphi_\rho|\le \frac{T^2}{2\alpha^3}{\rm e}^{2\rho \|\psi\|_{\infty}}|\nabla_x\varphi_\rho|^3,
\label{lemma-stima-2bis}
\\[2mm]
&|D_t\nabla_x\varphi_{\rho}|\le \frac{T^3}{4\a^2}|\nabla_x\varphi_{\rho}|^3,
\label{lemma-stima-3}
\\[2mm]
&|D_{x_ix_j}\varphi_\rho| \le \frac{T^2}{4\alpha^2}\left (\|\psi\|_{2,\infty}+\alpha
\|\psi\|_{1,\infty}\right )|\nabla_x\varphi_{\rho}|^2\no\\
&\phantom{|D_{x_ix_j}\varphi_\rho|}
\le \frac{T^4}{16\alpha^3}\left (\|\psi\|_{2,\infty}+\alpha
\|\psi\|_{1,\infty}\right )|\nabla_x\varphi_{\rho}|^3,
\label{lemma-stima-5}
\\[2mm]
&|D_{x_ix_jx_k}\varphi_\rho| \le
\frac{T^2}{4\alpha^2}\left (\|\psi\|_{3,\infty}+3\alpha\|\psi\|_{2,\infty}
+\alpha^2\rho\|\psi\|_{1,\infty}\right )|\nabla_x\varphi_{\rho}|^2\no\\
&\phantom{|D_{x_ix_jx_k}\varphi_\rho|}\le
\frac{T^4}{16\alpha^3}\left (\|\psi\|_{3,\infty}+3\alpha\|\psi\|_{2,\infty}+\alpha^2\|\psi\|_{1,\infty}\right )|\nabla_x\varphi_{\rho}|^3,
\label{lemma-stima-5bis}
\end{align}
for any $i,j,k=1,\ldots,M$.
\end{lemma}

\begin{proof}
The proof of (\ref{lemma-stima-1}) is straightforward. We limit ourselves to proving (\ref{lemma-stima-2}), (\ref{lemma-stima-5}) and
(\ref{lemma-stima-5bis}), the other estimates being completely similar to prove.

Since
\begin{eqnarray*}
D_t\varphi_{\rho}(t,\cdot)=(t-2T)\frac{{\rm e}^{2\rho\|\psi\|_{\infty}}-{\rm e}^{\rho\psi}}{(\ell(t))^2},\qquad\;\,t\in (0,T),
\end{eqnarray*}
we can estimate, using (\ref{lemma-stima-1}),
\begin{eqnarray*}
|D_t\varphi_{\rho}(t,\cdot)|\le \frac{T}{(\ell(t))^2}{\rm e}^{2\rho\|\psi\|_{\infty}}
\le \frac{T}{\alpha^2\rho^2}{\rm e}^{2\rho\|\psi\|_{\infty}}|\nabla_x\varphi_{\rho}|^2,
\end{eqnarray*}
for any $t\in (0,T)$,
which gives the first inequality in (\ref{lemma-stima-2}) since $\rho\ge 1$.

To prove the second inequality in (\ref{lemma-stima-2}) it suffices to use again
(\ref{lemma-stima-1}) and the estimate $\|\ell\|_{\infty}\le T^2/4$ to obtain
\begin{eqnarray*}
\frac{T}{\alpha^2}{\rm e}^{2\rho\|\psi\|_{\infty}}|\nabla_x\varphi_{\rho}|^2
=\frac{T}{\alpha^2}{\rm e}^{2\rho\|\psi\|_{\infty}}\frac{\ell(t)}{\alpha\rho} \frac{\alpha\rho}{\ell(t)}|\nabla_x\varphi_{\rho}|^2
\le \frac{T^3}{4\alpha^3}{\rm e}^{2\rho\|\psi\|_{\infty}}|\nabla_x\varphi_{\rho}|^3.
\end{eqnarray*}

Let us finally prove the first inequalities in (\ref{lemma-stima-5}) and (\ref{lemma-stima-5bis}),
the other two inequalities in (\ref{lemma-stima-5}) and (\ref{lemma-stima-5bis})
then follow from these ones and (\ref{lemma-stima-1}).
For this purpose we observe that
\begin{eqnarray*}
D_{x_ix_j}\varphi_{\rho}=\frac{\rho}{\ell}(D_{x_ix_j}\psi+\rho D_{x_i}\psi D_{x_j}\psi){\rm e}^{\rho\psi}.
\end{eqnarray*}
Hence,
\begin{align*}
|D_{x_ix_j}\varphi_{\rho}|&\le\frac{\rho}{\ell}\|\psi\|_{2,\infty}{\rm e}^{\rho\psi}+\rho\|\psi\|_{1,\infty}|\nabla_x\varphi_{\rho}|\\
&=\frac{\rho}{\ell}\frac{\ell^2}{\alpha^2\rho^2}\frac{\alpha^2\rho^2}{\ell^2}{\rm e}^{\rho\psi}\|\psi\|_{2,\infty}+
\rho \frac{\ell}{\alpha\rho}\frac{\alpha\rho}{\ell}\|\psi\|_{1,\infty}|\nabla_x\varphi_{\rho}|\\
&\le\frac{\ell}{\alpha^2\rho}\left (\frac{\rho^2}{\ell^2}|\nabla_x\psi|^2{\rm e}^{2\rho\psi}\right )\|\psi\|_{2,\infty}+
\frac{\ell}{\alpha}\frac{\alpha\rho}{\ell}\|\psi\|_{1,\infty}|\nabla_x\varphi_{\rho}|\\
&\le\frac{T^2}{4\alpha^2}\|\psi\|_{2,\infty}|\nabla_x\varphi_{\rho}|^2+\frac{T^2}{4\alpha}
\|\psi\|_{1,\infty}|\nabla_x\varphi_{\rho}|^2.
\end{align*}
The first estimate in (\ref{lemma-stima-5}) follows at once.

Similarly, one has
\begin{align*}
D_{x_ix_jx_k}\varphi_{\rho}=\frac{\rho}{\ell}[&D_{x_ix_jx_k}\psi+\rho (D_{x_ix_j}\psi D_{x_k}\psi+D_{x_jx_k}\psi D_{x_i}\psi+D_{x_ix_k}\psi D_{x_j}\psi)\\
&\;\,+\rho^2D_{x_i}\psi D_{x_j}\psi D_{x_k}\psi ]{\rm e}^{\rho\psi}.
\end{align*}
Hence, arguing as above, one gets
\begin{align*}
|D_{x_ix_jx_k}\varphi_{\rho}|\le&\frac{\rho}{\ell}\|\psi\|_{3,\infty}{\rm e}^{\rho\psi}+3\rho\|\psi\|_{2,\infty}|\nabla_x\varphi_{\rho}|
+\rho^2\|\psi\|_{1,\infty}|\nabla_x\psi||\nabla_x\varphi_{\rho}|\\
\le&\frac{\ell}{\alpha^2\rho}\left (\frac{\rho^2}{\ell^2}|\nabla_x\psi|^2{\rm e}^{2\rho\psi}\right )\|\psi\|_{3,\infty}+3\frac{\ell}{\alpha}\frac{\alpha\rho}{\ell}\|\psi\|_{2,\infty}|\nabla_x\varphi_{\rho}|\\
&+\|\psi\|_{1,\infty}\left (\frac{\rho}{\ell}|\nabla_x\psi|{\rm e}^{\rho\psi}\right )\rho\ell|\nabla_x\varphi_{\rho}|
\\
\le&\frac{T^2}{4\alpha^2}\|\psi\|_{3,\infty}|\nabla_x\varphi_{\rho}|^2+\frac{3T^2}{4\alpha}\|\psi\|_{2,\infty}|\nabla_x\varphi_{\rho}|^2
+\frac{T^2}{4}\rho\|\psi\|_{1,\infty}|\nabla_x\varphi_{\rho}|^2.
\end{align*}
\end{proof}

The main result of this section is the following theorem.

\begin{theorem}[Carleman estimates]
\label{thm-main1}
There exist two positive constants $\widetilde \lambda_0=\widetilde \lambda_0(a_0,
\|a\|_{{2,\infty}},\|b_0\|_{\infty},\|{\rm div}\,b\|_{\infty},\|c\|_{\infty},\|\psi\|_{3,\infty},\alpha,T)$, $\rho_0=\rho_0(a_0,\|a\|_{{1,\infty}},\|\psi\|_{2,\infty},\alpha)$ for which the following estimate holds for all $\l\ge \widetilde \l_0$ and all $v\in H^1((0,T);L^2_{\C}(\Omega\times{\mathcal O}))\cap L^2((0,T);{\mathcal H}^2_{\C}(\Om\times \R^N))$:
\begin{equation}
\int_{Q_T} \big (\l|\nabla_x\varphi_{\rho_0}||\nabla_xv|^2 + \l^3|\nabla_x\varphi_{\rho_0}|^3v^2\big)
{\rm e}^{2\l \varphi_{\rho_0}}\,\rmd t\rmd x\rmd y \le 64\int_{Q_T} |{\mathcal P}v|^2{\rm e}^{2\l \varphi_{\rho_0}}\,\rmd t\rmd x\rmd y.
\label{6.72}
\end{equation}
Here, ${\mathcal P}v=D_tv-{\rm div}_x(a\nabla_xv)-b\cdot\nabla_yv-c\cdot\nabla_xv-b_0v$.
\end{theorem}

\begin{corollary}
Suppose that $g$ and $h$ vanish in $Q_T$. If $u$ is a solution to problem $\eqref{5.1}$ in $Q_T$, then $u=0$.
\end{corollary}

\begin{proof}
If $g=h=0$ in $Q_T$, then ${\widetilde g}=0$ in $Q_T$. Estimate (\ref{6.72}) yields $v=0$ in $Q_T$.
This, in turn, implies, via the equality $v=u-h$, $u=0$ in $Q_T$, so that the principle of {\it unique  continuation} holds for the solution to problem (\ref{5.1}).
\end{proof}

The proof of Theorem \ref{thm-main1} is a consequence of the following theorem for real-valued functions and the principal part ${\mathcal P}_0$ of operator ${\mathcal P}$.

\begin{theorem}[Carleman estimates in a simplified case]
\label{thm-main2}
Two positive constants $\lambda_0=\lambda_0(a_0,\|a\|_{2,\infty},\|{\rm div}\,b\|_{\infty},\|\psi\|_{3,\infty},\alpha,T)$ and $\rho_0=\rho_0(a_0,\|a\|_{1,\infty},\|\psi\|_{2,\infty},\alpha)$ exist such that
\begin{equation}
\int_{Q_T} \big (\l|\nabla_x\varphi_{\rho_0}||\nabla_xv|^2 + \l^3|\nabla_x\varphi_{\rho_0}|^3v^2\big)
{\rm e}^{2\l \varphi_{\rho_0}}\rmd t\rmd x\rmd y
\le \frac{32}{3}\int_{Q_T} |{\mathcal P}_0v|^2{\rm e}^{2\l \varphi_{\rho_0}}\rmd t\rmd x\rmd y,
\label{6.73}
\end{equation}
for all $v\in H^1((0,T);L^2(\Omega\times{\mathcal O}))\cap L^2((0,T);{\mathcal H}^2(\Om\times{\mathcal O}))$ and all $\l\ge\l_0$.
Here, ${\mathcal P}_0v=D_tv-{\rm div}_x(a\nabla_xv)-b\cdot\nabla_yv$.
\end{theorem}

\begin{proof}[Proof of Theorem $\ref{thm-main1}$]
Fix $v\in H^1((0,T);L^2_{\C}(\Omega\times{\mathcal O}))\cap L^2((0,T);{\mathcal H}^2_{\C}(\Om\times{\mathcal O}))$. Then, $v_1={\rm Re}\,v$ and
$v_2={\rm Im}\,v$ belong to the space $H^1((0,T);L^2(\Omega\times{\mathcal O}))\cap L^2((0,T);
{\mathcal H}^2(\Om\times{\mathcal O}))$. Consequently, both $v_1$ and $v_2$ satisfy estimate (\ref{6.73}),
where $v$ is replaced with $v_j$, $j=1,2$. Since the coefficients of the operator ${\mathcal P}_0$ are all real-valued functions,
$|{\mathcal P}_0v|^2=|{\mathcal P}_0v_1|^2+|{\mathcal P}_0v_2|^2$. Therefore, summing the Carleman estimates for $v_1$ and $v_2$, we get
(\ref{6.73}) for $v$.

To show that $v$ satisfies (\ref{6.72}), we take advantage of the elementary inequalities
\begin{eqnarray*}
|{\mathcal P}_0v|^2\le 3|{\mathcal P}v|^2+ 3\|c\|^2_{\infty}|\nabla_xv|^2+3\|b_0\|_{\infty}^2|v|^2
\end{eqnarray*}
and of (\ref{lemma-stima-1}), with $\rho=\rho_0$, implying
\begin{align*}
& \int_{Q_T} \big (\|c\|^2_{\infty}|\nabla_xv|^2 + \|b_0\|^2_{\infty}|v|^2\big ){\rm e}^{2\l \varphi_{\rho_0}}\,\rmd t\rmd x\rmd y
\no\\[2mm]
\le &\frac{T^2}{4\a\rho_0}\|c\|^2_{\infty}\int_{Q_T} |\nabla_x\varphi_{\rho_0}||\nabla_xv|^2{\rm e}^{2\l \varphi_{\rho_0}}\,\rmd t\rmd x\rmd y
\no\\[2mm]
& + \frac{T^6}{64\a^3\rho_0^3}\|b_0\|^2_{\infty}\int_{Q_T} |\nabla_x\varphi_{\rho_0}|^3|v|^2
{\rm e}^{2\l \varphi_{\rho_0}}\,\rmd t\rmd x\rmd y.
\end{align*}
Choose now
\begin{eqnarray*}
\l\ge \frac{16T^2}{\a\rho_0}\|c\|^2_{\infty},\qquad\;\,
\l^3\ge \frac{T^6}{\a^3\rho_0^{3}}\|b_0\|^2_{\infty}.
\end{eqnarray*}
Then, (\ref{6.72}) holds with $\widetilde\l_0$ being defined by
\begin{eqnarray*}
{\widetilde \l_0}=\max\Big\{\l_0,\frac{16T^2}{\a\rho_0}\|c\|^2_{\infty},
\frac{T^2}{\a\rho_0}\|b\|^{2/3}_{\infty}\Big\}.
\end{eqnarray*}
\end{proof}

\begin{proof}[Proof of Theorem $\ref{thm-main2}$]
Let $w_\rho:Q_T\to\R$ be the function defined by
\begin{eqnarray}
w_{\rho}(t,x,y)=v(t,x,y){\rm e}^{\l \varphi_\rho(t,x)},\qquad\;\, (t,x,y)\in (0,T)\times\Omega\times {\mathcal O},
\label{def-w}
\end{eqnarray}
depending on the positive parameters $\l$ and $\rho$. According to the definitions of $\varphi_{\rho}$ we easily deduce that $w_{\rho}$ has the same degree of smoothness as $v$. Moreover, $\|w_{\rho}(t,\cdot)\|_{H^2(\Omega\times{\mathcal O})}$ and $w_{\rho}(t,x,y)$ tend to $0$ as $t\to 0^+$ and $t\to T^-$, the latter one for any $(x,y)\in\Omega\times\mathcal{O}$.

For almost all the proof, to avoid cumbersome notation, we simply write $w$ and $\varphi$ instead of $w_\rho$ and $\varphi_{\rho}$.

Define the linear operator $\mathcal{L}_\l$ by
\begin{equation}
{\mathcal L}_\l w={\rm e}^{\l \varphi}{\mathcal P}_0(w{\rm e}^{-\l \varphi}).
\label{6.6}
\end{equation}
After some computations we can split ${\mathcal L}_{\lambda}$ into the sum ${\mathcal L}_\l={\mathcal L}^+_{1,\l}+({\mathcal L}^-_{1,\l}-b\cdot\nabla_y)$,
where
\begin{align}
&{\mathcal L}^+_{1,\l}w = - {\rm div}_x(a\nabla_xw) -\l(D_t\varphi+\l a|\nabla_x \varphi|^2)w
=:\sum_{k=1}^2{\mathcal L}^+_{1,\l,k} w,
\label{6.10}
\\[2mm]
&{\mathcal L}^-_{1,\l}w = D_tw + 2\l a\nabla_x\varphi\cdot\nabla_x w
+ \l(a\Delta_x\varphi+\nabla_xa\cdot\nabla_x\varphi)w
=: \sum_{k=1}^3{\mathcal L}^-_{1,\l,k} w.
\label{6.11}
\end{align}
Clearly,
\begin{align}
\|{\mathcal L}_\l w\|^2_2&=\|{\mathcal L}^+_{1,\l}w\|^2_2
+\|{\mathcal L}^-_{1,\l}w-b\cdot\nabla_yw\|^2_2
+2({\mathcal L}^+_{1,\l}w,{\mathcal L}^-_{1,\l}w-b\cdot\nabla_yw)_2
\no\\[1mm]
&\ge \|{\mathcal L}^+_{1,\l}w\|^2_2 + 2({\mathcal L}^+_{1,\l}w,{\mathcal L}^-_{1,\l}w)_2
- 2({\mathcal L}^+_{1,\l}w,b\cdot\nabla_yw)_2.
\label{6.12}
\end{align}

To rewrite the terms $({\mathcal L}^+_{1,\l}w,{\mathcal L}^-_{1,\l}w)_2$ and $({\mathcal L}^+_{1,\l}w,b\cdot\nabla_yw)_2$ in a more convenient way,
we perform several integrations by parts.

As the rest of the proof is rather long, we split it into five steps and, for notational convenience, we set
\begin{eqnarray}
\label{I12}
{\mathcal I}_1(w)=\int_{Q_T}|\nabla_x\varphi|^3w^2\rmd t\rmd x\rmd y,\qquad\;\,
{\mathcal I}_2(w)=\int_{Q_T}|\nabla_x\varphi||\nabla_xw|^2\rmd t\rmd x\rmd y.
\end{eqnarray}
Moreover, we denote by $C_j$ positive constants which depend only on the quantities in brackets.
\smallskip

{\em Step 1: the term $2({\mathcal L}^+_{1,\l}w,{\mathcal L}^-_{1,\l}w)_2$.} We split this term into the sum
of the addenda $({\mathcal L}^+_{1,\l,i}w,{\mathcal L}^-_{1,\l,j}w)_2$ ($i=1,2$, $j=1,2,3$).

We claim that $({\mathcal L}^+_{1,\l,1}w,{\mathcal L}^-_{1,\l,1}w)_2=0$. To prove the claim, we need to integrate by parts. For this purpose, we approximate function $v$ by a sequence $\{v_n\}\subset H^1((0,T);L^2({\mathcal O};H^2(\Om)\cap H^1_0(\Om)))$, converging to $v$ in $L^2(Q_T)$, together with its first-order time derivative and first- and second-order spatial derivatives with respect to $x$.
Set $w_n={\rm e}^{\lambda\varphi}v_n$ and observe that, integrating by parts with respect to the variable $t$ and recalling that, for any $n\in\N$, $w_n=0$ on $(0,T)\times \partial \Om\times {\mathcal O}$ and $\nabla_x w_n(t,\cdot,\cdot)$ tends to $0$ in $(L^2(\Om\times{\mathcal O}))^N$ as $t\to 0^+$ and $t\to T^-$, we easily see that
\begin{align*}
\int_{Q_T} {\rm div}_x(a\nabla_xw_n)D_tw_n \rmd t\rmd x\rmd y=&
-\int_{Q_T} a\nabla_xw\cdot D_t\nabla_x w_n \rmd t\rmd x\rmd y\\
=&-\frac{1}{2} \int_{Q_T} aD_t|\nabla_x w_n|^2 \rmd t\rmd x\rmd y\\
=&0.
\end{align*}
Letting $n\to +\infty$ gives
\begin{eqnarray*}
\int_{Q_T} {\rm div}_x(a\nabla_xw)D_tw\rmd t\rmd x\rmd y=0.
\end{eqnarray*}

As far as the term $2({\mathcal L}^+_{1,\l,2}w,{\mathcal L}^-_{1,\l,1}w)_2$ is concerned, we observe that
\begin{align*}
2({\mathcal L}^+_{1,\l,2}w,{\mathcal L}^-_{1,\l,1}w)_2
&=-2\l\int_{Q_T}(D_t\varphi+\l a|\nabla_x \varphi|^2)w D_tw\, \rmd t\rmd x\rmd y
\no\\
&=-\l\int_{Q_T}(D_t\varphi+\l a|\nabla_x \varphi|^2)D_t(w^2) \rmd t\rmd x\rmd y
\no\\
&=\l\int_{Q_T} w^2(D^2_t\varphi+\l aD_t|\nabla_x \varphi|^2)\rmd t\rmd x\rmd y.
\end{align*}

Computing the terms $2({\mathcal L}^+_{1,\l,1}w,{\mathcal L}^-_{1,\l,2}w)_2$ is much trickier.
Integrating twice by parts, we can write
\begin{align}
2({\mathcal L}^+_{1,\l,1}w,{\mathcal L}^-_{1,\l,2}w)_2=&
-4\lambda\int_{(0,T)\times\partial\Omega\times{\mathcal O}}
aD_{\nu}w(a\nabla_x\varphi\cdot\nabla_xw)\rmd t\rmd\sigma(x)\rmd y\no\\
&+2\l\int_{Q_T} a^2\nabla_x\varphi\cdot\nabla_x(|\nabla_xw|^2)\rmd t\rmd x\rmd y
\no\\
&+4\l\int_{Q_T} a\sum_{j,k=1}^{M}\,D_{x_j}(aD_{x_k}\varphi)(D_{x_k}w D_{x_j}w)\rmd t\rmd x\rmd y\no\\
=&-4\lambda\int_{(0,T)\times\partial\Omega\times{\mathcal O}}
aD_{\nu}w(a\nabla_x\varphi\cdot\nabla_xw)\rmd t\rmd\sigma(x)\rmd y\no\\
&+2\l\int_{(0,T)\times\partial\Omega\times{\mathcal O}} a^2D_{\nu}\varphi |\nabla_xw|^2\rmd t\rmd\sigma(x)\rmd y\no\\
&-2\l\int_{Q_T}{\rm div}_x(a^2\nabla_x\varphi)|\nabla_xw|^2\rmd t\rmd x\rmd y\no\\
&+4\l\int_{Q_T} a(\nabla_xa\cdot\nabla_xw)(\nabla_x\varphi\cdot\nabla_xw)\rmd t\rmd x\rmd y
\no\\
&+4\l\int_{Q_T} a^2\sum_{j,k=1}^{M}D_{x_j}D_{x_k}\varphi(D_{x_k}w D_{x_j}w)\rmd t\rmd x\rmd y.
\label{AAA}
\end{align}
To rewrite the integrals on $(0,T)\times\partial\Omega\times{\mathcal O}$, we observe that,
since $w\equiv 0$ on $(0,T)\times\partial \Om\times {\mathcal O}$, $\nabla_xw=D_{\nu}w\,{\boldsymbol\nu}$ on $(0,T)\times\partial \Om\times{\mathcal O}$. Therefore,
$\nabla_x\varphi\cdot \nabla_x w = \rho \ell^{-1}{\rm e}^{\rho\psi}D_{\nu}\psi D_{\nu}w$ on the same set.
Moreover, $\nabla_x w=0$ on $(0,T)\times\Gamma\times{\mathcal O}$ and $D_{\nu}\varphi=\rho\ell^{-1}{\rm e}^{\rho\psi}D_{\nu}\psi \le 0$ on $(0,T)\times(\partial \Om\setminus \Gamma)\times{\mathcal O}$.
We can thus write
\begin{align}
&-4\lambda\int_{(0,T)\times\partial\Omega\times{\mathcal O}}
aD_{\nu}w(a\nabla_x\varphi\cdot\nabla_xw)\rmd t\rmd\sigma(x)\rmd y\no\\
=&4\rho\lambda\int_{(0,T)\times(\partial\Omega\setminus\Gamma)\times{\mathcal O}}
\frac{a^2}{\ell}|D_{\nu}\psi|(D_{\nu}w)^2{\rm e}^{\rho\psi}\rmd t\rmd\sigma(x)\rmd y
\label{BBB}
\end{align}
and
\begin{align}
&2\l\int_{(0,T)\times\partial\Omega\times{\mathcal O}} a^2D_{\nu}\varphi |\nabla_xw|^2\rmd t\rmd\sigma(x)\rmd y\no\\
=-&2\rho\l\int_{(0,T)\times(\partial\Omega\setminus\Gamma)\times{\mathcal O}} \frac{a^2}{\ell}|D_{\nu}\psi|(D_{\nu}w)^2{\rm e}^{\rho\psi}\rmd t\rmd\sigma(x)\rmd y.
\label{CCC}
\end{align}
From (\ref{AAA}), (\ref{BBB}) and (\ref{CCC}) we get
\begin{align*}
2({\mathcal L}^+_{1,\l,1}w,{\mathcal L}^-_{1,\l,2}w)_2=&
2\rho\l\int_{(0,T)\times(\partial\Omega\setminus\Gamma)\times{\mathcal O}} \frac{a^2}{\ell}|D_{\nu}\psi||D_{\nu}w|^2{\rm e}^{\rho\psi}\rmd t\rmd\sigma(x)\rmd y\no\\
&-2\l\int_{Q_T}{\rm div}_x(a^2\nabla_x\varphi)|\nabla_xw|^2\rmd t\rmd x\rmd y\no\\
&+4\l\int_{Q_T} a(\nabla_xa\cdot\nabla_xw)(\nabla_x\varphi\cdot\nabla_xw)\rmd t\rmd x\rmd y
\no\\
&+4\l\int_{Q_T} a^2\sum_{j,k=1}^{M}D_{x_j}D_{x_k}\varphi(D_{x_k}w D_{x_j}w)\rmd t\rmd x\rmd y.
\end{align*}

Further, straightforward integrations by parts, where we take into account that $w$ vanishes on $(0,T)\times\partial\Omega\times{\mathcal O}$, show that
\begin{align*}
2({\mathcal L}^+_{1,\l,1}w,{\mathcal L}^-_{1,\l,3}w)_2
=&2\l \int_{Q_T} a(a\Delta_x\varphi+\nabla_xa\cdot\nabla_x\varphi)|\nabla_xw|^2\rmd t\rmd x\rmd y
\no\\
&+2\l\int_{Q_T} aw\nabla_xw\cdot\nabla_x(a\Delta_x\varphi+\nabla_xa\cdot\nabla_x\varphi)\rmd t\rmd x\rmd y
\end{align*}
and
\begin{align*}
2({\mathcal L}^+_{1,\l,2}w,{\mathcal L}^-_{1,\l,2}w)_2
=& -2\lambda^2\int_{Q_T}
a(D_t\varphi+\l a|\nabla_x\varphi|^2)\nabla_x\varphi\cdot \nabla_x (w^2)\,\rmd t\rmd x\rmd y
\no\\
=& 2\lambda^2\int_{Q_T}
w^2{\rm div}_x[a(D_t\varphi+\l a|\nabla_x\varphi|^2)\nabla_x\varphi]\,\rmd t\rmd x\rmd y\\
=&2\lambda^2\int_{Q_T}w^2(D_t\varphi+\l a|\nabla_x\varphi|^2){\rm div}_x(a\nabla_x\varphi)\,\rmd t\rmd x\rmd y\\
&+2\lambda^2\int_{Q_T}w^2a\nabla_x\varphi\cdot\nabla_x(D_t\varphi+\l a|\nabla_x\varphi|^2)\,\rmd t\rmd x\rmd y.
\end{align*}
Finally, we have
\begin{eqnarray*}
2({\mathcal L}^+_{1,\l,2}w,{\mathcal L}^-_{1,\l,3}w)_2
= -2\l^2\int_{Q_T}  w^2(D_t\varphi+\l a|\nabla_x \varphi|^2){\rm div}_x(a\nabla_x\varphi)\rmd t\rmd x\rmd y.
\end{eqnarray*}
Summing the previous formulas we get
\begin{align}
2({\mathcal L}^+_{1,\l}w,{\mathcal L}^-_{1,\l}w)_2=&
\lambda\int_{Q_T}{\mathcal H}_1(\varphi)w^2\rmd t\rmd x\rmd y
+\l^2\int_{Q_T} {\mathcal H}_2(a,\varphi)w^2\rmd t\rmd x\rmd y
\no\\
&+\l^3\int_{Q_T} {\mathcal H}_3(a,\varphi)w^2\rmd t\rmd x\rmd y
+2\l{\mathcal K}(a,\varphi,w)\no\\
&+2\rho\l\int_{(0,T)\times(\partial \Om\setminus\Gamma)\times{\mathcal O}} \frac{a^2}{\ell}|D_{\nu}\psi|(D_{\nu}w)^2e^{\rho\psi}\rmd td\s(x)\rmd y,
\label{J2}
\end{align}
where
\begin{align}
\label{H1}
&{\mathcal H}_1(\varphi) =D^2_t\varphi,
\\[3mm]
\label{H2}
&{\mathcal H}_2(a,\varphi)=4a\nabla_x D_t\varphi\cdot\nabla_x\varphi,
\\[3mm]
\label{H3}
&{\mathcal H}_3(a,\varphi)=2a|\nabla_x \varphi|^2\nabla_xa\cdot\nabla_x\varphi
+2a^2\nabla_x\varphi\cdot\nabla_x(|\nabla_x\varphi|^2),
\\[3mm]
\label{H4}
&{\mathcal K}(a,\varphi,w)=
-\int_{Q_T} (a\nabla_xa\cdot\nabla_x\varphi)|\nabla_xw|^2\rmd t\rmd x\rmd y
\no\\
&\phantom{{\mathcal K}(a,\varphi,\nabla_xw)=}
+2\int_{Q_T} a(\nabla_xa\cdot\nabla_xw)(\nabla_x\varphi\cdot\nabla_xw)\rmd t\rmd x\rmd y
\no\\
&\phantom{{\mathcal K}(a,\varphi,\nabla_xw)=}
+2\int_{Q_T} a^2\sum_{j,k=1}^{M}\,D_{x_j}D_{x_k}\varphi (D_{x_k}w D_{x_j}w)\rmd t\rmd x\rmd y\no\\
&\phantom{{\mathcal K}(a,\varphi,\nabla_xw)=}
+\int_{Q_T} aw\nabla_xw\cdot\nabla_x(a\Delta_x\varphi+\nabla_xa\cdot\nabla_x\varphi) \rmd t\rmd x\rmd y
\end{align}

{\em Step 2: the term $-2({\mathcal L}^+_{1,\l}w,b\cdot\nabla_yw)_2$.} Since
\begin{eqnarray*}
(a\nabla_xw)\cdot \nabla_x(b\cdot\nabla_yw)
= \frac{1}{2} ab\cdot \nabla_{y}|\nabla_x w|^2,
\end{eqnarray*}
recalling that $w=0$ on $[0,T]\times\partial_*(\Omega\times{\mathcal O})$ implies $\nabla_x w=0$ on $[0,T]\times \Omega\times \partial{\mathcal O}$ and $\nabla_y w=0$ on $[0,T]\times\partial \Omega\times {\mathcal O}$, an integration by parts shows that
\begin{align}
-2({\mathcal L}^+_{1,\l}w,b\cdot\nabla_yw)_2=&
2\int_{Q_T} {\rm div}_x(a\nabla_xw)(b\cdot\nabla_yw)\rmd t\rmd x\rmd y
\no\\
&+2\l\int_{Q_T}(D_t\varphi+\l a|\nabla_x \varphi|^2)w (b\cdot\nabla_yw)\rmd t\rmd x\rmd y
\no\\
=&-\int_{Q_T} ab\cdot\nabla_y(|\nabla_xw|^2) \rmd t\rmd x\rmd y\no\\
&-\l\int_{Q_T} ({\rm div}\,b)(D_t\varphi+\l a|\nabla_x\varphi|^2)w^2\rmd t\rmd x\rmd y
\no\\
=&\int_{Q_T} a({\rm div}\,b)|\nabla_xw|^2\rmd t\rmd x\rmd y\no\\
&-\lambda\int_{Q_T} ({\rm div}\,b)(D_t\varphi+ \l a|\nabla_x \varphi|^2)w^2\rmd t\rmd x\rmd y.
\label{J3}
\end{align}

Summing up, from formulae (\ref{6.12}), (\ref{J2}) and (\ref{J3}), we obtain the following estimate from below for the norm of ${\mathcal L}_\l w$:
\begin{equation}
\|{\mathcal L}_\l w\|^2_2\ge
{\mathcal J}_1(w)+\|{\mathcal L}_{1,\lambda}^+w\|_2^2,
\label{split-1}
\end{equation}
where
\begin{align*}
{\mathcal J}_1(w)=& {\ov {\mathcal K}}(a,\varphi,w,\lambda)
+\lambda\int_{Q_T}{\ov {\mathcal H}}_1(b,\varphi)w^2\,\rmd t\rmd x\rmd y\\
&+\lambda^2\int_{Q_T}{\ov {\mathcal H}}_2(a,b,\varphi)w^2\,\rmd t\rmd x\rmd y
+\lambda^3\int_{Q_T}{\ov {\mathcal H}}_3(a,\varphi)w^2\,\rmd t\rmd x\rmd y
\end{align*}
and
\begin{align}
& {\ov{\mathcal H}}_1(b,\varphi)={\mathcal H}_1(\varphi)-({\rm div}\,b)D_t\varphi,
\label{A1}
\\
& {\ov{\mathcal H}}_2(a,b,\varphi)={\mathcal H}_2(a,\varphi)-a({\rm div}\,b)|\nabla_x\varphi|^2,
\label{A2}
\\
& {\ov{\mathcal H}}_3(a,\varphi)={\mathcal H}_3(a,\varphi),
\label{A3}
\\
& {\ov {\mathcal K}}(a,b,\varphi,w,\lambda)=2\l{\mathcal K}(a,\varphi,w)
+\int_{Q_T}a({\rm div}\,b)|\nabla_xw|^2\rmd t\rmd x\rmd y.
\label{A4}
\end{align}

{\em Step 3: estimate of ${\mathcal J}_1(w)$.}
As a first step, taking advantage of the formula
\begin{eqnarray*}
D_{x_k}D_{x_j}\varphi =\frac{\rho}{\ell}{\rm e}^{\rho\psi}\big (D_{x_k}D_{x_j}\psi+\rho D_{x_j}\psi D_{x_k}\psi
\big),
\end{eqnarray*}
we obtain (cf. (\ref{form-phi})
\begin{align*}
\sum_{j,k=1}^{M}D_{x_j}D_{x_k}\varphi (D_{x_k}w D_{x_j}w)
=& \frac{\rho}{\ell}{\rm e}^{\rho\psi}\sum_{j,k=1}^{M}D_{x_j}D_{x_k}\psi(D_{x_k}w D_{x_j}w)\\
&+\frac{\rho^2}{\ell}{\rm e}^{\rho\psi}(\nabla\psi\cdot\nabla w)^2\\
=& |\nabla_x \psi|^{-1}|\nabla_x \varphi|\sum_{j,k=1}^{M}D_{x_j}D_{x_k}\psi(D_{x_k}w D_{x_j}w)\\
&+\frac{\rho^2}{\ell}{\rm e}^{\rho\psi}(\nabla\psi\cdot\nabla w)^2\\
\ge &  -\frac{M}{\alpha}\|\psi\|_{2,\infty}|\nabla_x\varphi||\nabla w|^2,
\end{align*}
where $\alpha$ is the infimum of the function $|\nabla\psi|$ over $\Omega$.
Since
\begin{align*}
|a\nabla_x(a\Delta_x\varphi+\nabla_xa\cdot\nabla_x\varphi)|
\le &\|a\|_{\infty}\Bigg (\sqrt{M}\|a\|_{1,\infty}|\Delta_x\varphi|+
\|a\|_{\infty}|\nabla_x\Delta_x\varphi|\\
&\qquad\;\;\;\;\;\;+M\|a\|_{2,\infty}|\nabla_x\varphi|+M^{3/2}\|a\|_{1,\infty}\sum_{i,j=1}^M|D_{ij}\varphi|\Bigg )\\
\le & C_1(\|a\|_{2,\infty},\|\psi\|_{3,\infty},\rho,\alpha,T)|\nabla_x\varphi|^2,
\end{align*}
where we used (\ref{lemma-stima-1}) and the first inequalities in (\ref{lemma-stima-5}) and (\ref{lemma-stima-5bis}),
we can estimate (using H\"older inequality)
\begin{align*}
&2\lambda\int_{Q_T} aw\nabla_xw\cdot\nabla_x(a\Delta_x\varphi+\nabla_xa\cdot\nabla_x\varphi) \rmd t\rmd x\rmd y\\
\ge &-2\lambda\int_{Q_T} |w||\nabla_xw| |a\nabla_x(a\Delta_x\varphi+\nabla_xa\cdot\nabla_x\varphi)|\rmd t\rmd x\rmd y\\
\ge &
-2C_1(\|a\|_{2,\infty},\|\psi\|_{3,\infty},\rho,\alpha,T)
\int_{Q_T} (\lambda^{3/4}|\nabla_x\varphi|^{3/2}|w|)(\lambda^{1/4}|\nabla_x\varphi|^{1/2}|\nabla_xw|)\rmd t\rmd x\rmd y\\
\ge &-C_1(\|a\|_{2,\infty},\|\psi\|_{3,\infty},\rho,\alpha,T)\lambda^{3/2}
{\mathcal I}_1(w)
-C_1(\|a\|_{2,\infty},\|\psi\|_{3,\infty},\rho,\alpha,T)\lambda^{1/2}
{\mathcal I}_2(w),
\end{align*}
where ${\mathcal I}_1(w)$ and ${\mathcal I}_2(w)$ are defined in (\ref{I12}).
We conclude that
\begin{align*}
2\lambda{\mathcal K}(a,\varphi,w)\ge&
-C_1(\|a\|_{2,\infty},\|\psi\|_{3,\infty},\rho,\alpha,T)\lambda^{3/2}
{\mathcal I}_1(w)\\
&-\left [C_2(\|a\|_{1,\infty},\|\psi\|_{2,\infty},\alpha)\lambda
+C_1(\|a\|_{2,\infty},\|\psi\|_{3,\infty},\rho,\alpha,T)\lambda^{1/2}\right ]
{\mathcal I}_2(w).
\end{align*}
Moreover, from (\ref{lemma-stima-1}), and recalling that $\rho\ge 1$, we get
\begin{align*}
\left |\int_{Q_T}a({\rm div}\,b)|\nabla_xw|^2\rmd t\rmd x\rmd y\right |
\le &\|a\|_{\infty}\|{\rm div}\,b\|_{\infty}\int_{Q_T}|\nabla_xw|^2\rmd t\rmd x\rmd y
\\
\le&
\frac{T^2}{4\alpha}\|a\|_{\infty}\|{\rm div}\,b\|_{\infty}{\mathcal I}_2(w).
\end{align*}
Therefore, it follows that
\begin{align}
{\ov{\mathcal K}}(a,b,\varphi,\nabla_xw,\lambda)
\ge &-C_1(\|a\|_{2,\infty},\|\psi\|_{3,\infty},\rho,\alpha,T)\lambda^{3/2}
{\mathcal I}_1(w)\no\\
&-\Big [C_2(\|a\|_{1,\infty},\|\psi\|_{2,\infty},\a)\lambda
+C_1(\|a\|_{2,\infty},\|\psi\|_{3,\infty},\rho,\alpha,T)\lambda^{1/2}\no\\
&\qquad\;+\frac{T^2}{4\alpha}\|a\|_{\infty}\|{\rm div}\,b\|_{\infty}\Big ]{\mathcal I}_2(w).
\label{6.38a}
\end{align}

We now consider the terms containing ${\ov{\mathcal H}}_1(b,\varphi)$,
${\ov{\mathcal H}}_2(a,b,\varphi)$ and ${\ov{\mathcal H}}_3(a,\varphi)$
(cf. (\ref{A1})-(\ref{A3})). From (\ref{lemma-stima-1}), the second inequality in (\ref{lemma-stima-2}), (\ref{lemma-stima-2bis}), (\ref{lemma-stima-3}) and definitions (\ref{H1})-(\ref{H4})), we deduce the pointwise inequalities
\begin{align}
&|{\ov{\mathcal H}}_1(b,\varphi)|
\le C_3(\|{\rm div}\,b\|_{\infty},
\rho,\alpha,T)|\nabla_x\varphi|^3,
\label{stima-I1}\\
&|{\ov{\mathcal H}}_2(a,b,\varphi)|
\le C_4(\|a\|_\infty,\|{\rm div}\,b\|_{\infty},\alpha,T)|\nabla_x\varphi|^3,
\label{stima-I2}
\end{align}
where we have used the condition $\rho\ge 1$, and (recalling that $|\nabla_x\psi|\ge\alpha$)
\begin{align}
{\ov{\mathcal H}}_3(a,\varphi)=&2a|\nabla_x \varphi|^2\nabla_xa\cdot\nabla_x\varphi
+2a^2\nabla_x\varphi\cdot\nabla_x(|\nabla_x\varphi|^2)
\no\\
=&2a|\nabla_x \varphi|^2(\nabla_xa\cdot\nabla_x\varphi)
+2a^2\rho^3\ell^{-3}{\rm e}^{3\rho\psi}\nabla_x\psi\cdot\nabla_x(|\nabla_x\psi|^2)
\no\\
&+4a^2\rho^4\ell^{-3}{\rm e}^{3\rho\psi}|\nabla_x\psi|^4\no\\
\ge &4a^2_0\alpha\rho |\nabla_x\varphi|^3-C_5(\|a\|_{1,\infty},\|\psi\|_{2,\infty},\alpha)|\nabla_x\varphi|^3.
\label{stima-K3}
\end{align}

Summing up, from (\ref{6.38a})-(\ref{stima-K3}), we get the following estimate from below for ${\mathcal J}_1(w)$:
\begin{align}
{\mathcal J}_1(w)\ge &
\Big\{\left [4a_0^2\alpha\rho-C_5(\|a\|_{1,\infty},\|\psi\|_{2,\infty},\alpha)\right ]\lambda^3\no\\
&\;\;-C_4(\|a\|_{\infty},\|{\rm div}\,b\|_{\infty},\alpha,T)\lambda^2
-C_1(\|a\|_{2,\infty},\|\psi\|_{3,\infty},\rho,\alpha,T)\lambda^{3/2}
\no\\
&\;\;-C_3(\|{\rm div}\,b\|_{\infty},\rho,\alpha,T)\lambda\Big\}{\mathcal I}_1(w)
\no\\
&-\Big (C_2(\|a\|_{1,\infty},\|\psi\|_{2,\infty},\alpha)\lambda
+C_1(\|a\|_{2,\infty},\|\psi\|_{3,\infty},\rho,\alpha,T)\lambda^{1/2}\no\\
&\quad\;\;\;\;\;+\frac{T^2}{4\alpha}\|a\|_{\infty}\|{\rm div}\,b\|_{\infty}\Big ){\mathcal I}_2(w).
\label{stima-J1}
\end{align}

{\em Step 4: estimate of $\|{\mathcal L}_{1,\lambda}^+ w\|_2^2$.}
Using the inequalities
\begin{align*}
({\rm div}_x(a\nabla_xw))^2
=&\big[{\mathcal L}^+_{1,\l} w +  \l(D_t\varphi+\l a|\nabla_x \varphi|^2) w\big]^2
\no\\
\le & 3({\mathcal L}^+_{1,\l} w)^2 + 3\l^2 (D_t\varphi)^2w^2
+3\l^4a^2|\nabla_x \varphi|^4w^2,
\end{align*}
(cf. (\ref{6.10})), $|\nabla_x\varphi|\ge 4\rho\alpha T^{-2}$ (which follows from (\ref{lemma-stima-1}) and the first inequality in (\ref{lemma-stima-2}), we can infer that
\begin{align}
\frac{T^2}{\rho\alpha\lambda}\|{\mathcal L}_{1,\lambda}^+w\|_2^2\ge
&\int_{Q_T} \l^{-1}|\nabla_x \varphi|^{-1}\left ({\rm div}_x(a\nabla_x w)\right )^2\rmd t\rmd x\rmd y
\no\\
&-\big [4\|a\|_{\infty}^2\l^3+ C_6(\|\psi\|_{\infty},\rho,\alpha,T)\lambda\big ]{\mathcal I}_1(w).
\label{stima-I3-1}
\end{align}
Now we want to show that the integral term in (\ref{stima-I3-1})
can be estimated from below by a positive constant times
${\mathcal I}_2(w)$ minus some terms which can be controlled
by means of ${\mathcal I}_2(w)$ and the good term in ${\mathcal J}_2(w)$. For this purpose,
we begin by observing that an integration by parts yields
\begin{align*}
\rho^{1/2}\l\int_{Q_T} a|\nabla_x \varphi||\nabla_xw|^2\rmd t\rmd x\rmd y
=& \rho^{1/2}\l\int_{Q_T} |\nabla_x \varphi| (a\nabla_xw)\cdot\nabla_xw\, \rmd t\rmd x\rmd y
\no\\
=&-\rho^{1/2}\l\int_{Q_T} |\nabla_x \varphi|w\,{\rm div}_x(a\nabla_xw)\rmd t\rmd x\rmd y
\no\\
&-\frac{1}{2}\rho^{1/2}\l\int_{Q_T}a\nabla_x(|\nabla_x \varphi|)\cdot\nabla_x(w^2)\rmd t\rmd x\rmd y
\no\\
=&-\rho^{1/2}\l\int_{Q_T} |\nabla_x \varphi|w\,{\rm div}_x(a\nabla_xw)\rmd t\rmd x\rmd y
\no\\
&+\frac{1}{2}\rho^{1/2}\l\int_{Q_T} w^2\,{\rm div}_x[a\nabla_x(|\nabla_x \varphi|)]\rmd t\rmd x\rmd y.
\end{align*}
Since
\begin{align*}
\rho^{1/2}\l|\nabla_x \varphi||w\,{\rm div}_x(a\nabla_xw)|
&= (\l|\nabla_x \varphi|)^{-1/2}|{\rm div}_x(a\nabla_xw)|
\rho^{1/2}(\l|\nabla_x \varphi|)^{3/2}|w|\no\\
&\le\frac{1}{4\varepsilon}\l^{-1}|\nabla_x \varphi|^{-1}[{\rm div}_x(a\nabla_xw)]^2
+ \varepsilon\rho\l^3|\nabla_x \varphi|^3w^2,
\end{align*}
for any $\varepsilon>0$, we get
\begin{align*}
&\rho^{1/2}\l\int_{Q_T} a|\nabla_x \varphi| |\nabla_xw|^2\rmd t\rmd x\rmd y\no\\
\le &\frac{1}{4\varepsilon}\int_{Q_T}\l^{-1}|\nabla_x \varphi|^{-1}[{\rm div}_x(a\nabla_xw)]^2\rmd t\rmd x\rmd y
+ \rho\varepsilon\l^3{\mathcal I}_1(w)
\no\\
&+\frac{1}{2}\rho^{1/2}\l\int_{Q_T} w^2{\rm div}_x[a\nabla_x(|\nabla_x \varphi|)]\rmd t\rmd x\rmd y
\end{align*}
or, equivalently,
\begin{align}
&\int_{Q_T} \l^{-1}|\nabla_x \varphi|^{-1}[{\rm div}_x(a\nabla_xw)]^2\rmd t\rmd x\rmd y\no\\
\ge &
4\varepsilon a_0\rho^{1/2}\l{\mathcal I}_2(w)
- 4\varepsilon^2\rho\l^3{\mathcal I}_1(w)-2\varepsilon\rho^{1/2}\l\int_{Q_T} w^2{\rm div}_x[a\nabla_x(|\nabla_x \varphi|)]\rmd t\rmd x\rmd y.
\label{stima-I3-2}
\end{align}
Using the second estimates in (\ref{lemma-stima-5}) and (\ref{lemma-stima-5bis}) we can estimate
\begin{eqnarray}
|{\rm div}_x[a\nabla_x(|\nabla_x \varphi|)]|\le C_7(\|a\|_{1,\infty},\|\psi\|_{3,\infty},\alpha,T)
|\nabla_x\varphi|^3.
\label{stima-divergenza}
\end{eqnarray}
Replacing (\ref{stima-I3-2}) and (\ref{stima-divergenza}) into (\ref{stima-I3-1}) and assuming that
\begin{equation}
\frac{\alpha\lambda}{T^2}\ge 1,
\label{cond-1}
\end{equation}
implying $T^2/(\rho\alpha\lambda)\le 1$ since $\rho\ge 1$, we get
\begin{align}
\|{\mathcal L}_{1,\lambda}^+w\|_2^2\ge &
4\varepsilon a_0\rho^{1/2}\l{\mathcal I}_2(w)
-\Big [4\varepsilon^2\rho\l^3+4\l^3\|a\|_{\infty}^2 +\varepsilon C_8(\|a\|_{1,\infty},\|\psi\|_{3,\infty},\alpha,T)
\lambda\no\\
&\qquad\qquad\quad\qquad\;\;\;\;\;
+C_6(\|\psi\|_{\infty},\rho,\alpha,T)\lambda\Big ]
{\mathcal I}_1(w).
\label{stima-I3}
\end{align}

{\em Step 5: the final step.}
Under condition (\ref{cond-1}), from formula (\ref{split-1}) and estimates (\ref{stima-J1}) and (\ref{stima-I3}) we obtain
\begin{align*}
\|{\mathcal L}_{\lambda}w\|_2^2\ge&
\Big\{\left [4(a_0^2\alpha-\varepsilon^2)\rho-C_5(\|a\|_{1,\infty},\|\psi\|_{2,\infty},\alpha)-4\|a\|_{\infty}^2\right ]\lambda^3\\
&\;\,-C_4(\|a\|_\infty,\|{\rm div}\,b\|_{\infty},\alpha,T)\lambda^2
-C_1(\|a\|_{2,\infty},\|\psi\|_{3,\infty},\rho,\alpha,T)\lambda^{3/2}\\
&\;\,-C_3(\|{\rm div}\,b\|_{\infty},\rho,\alpha,T)\lambda-\varepsilon C_8(\|a\|_{1,\infty},\|\psi\|_{3,\infty},\alpha,T)\lambda\\
&\;\,-C_6(\|\psi\|_{\infty},\rho,\alpha,T)\lambda\Big\}
{\mathcal I}_1(w)
\no\\
&+\bigg\{2\left [2\varepsilon a_0\rho^{1/2}
-C_2(\|a\|_{1,\infty},\|\psi\|_{2,\infty},\alpha)\right ]\lambda\no\\
&\quad\;\;\;\;-C_1(\|a\|_{2,\infty},\|\psi\|_{3,\infty},\rho,\alpha,T)\lambda^{1/2}
-\frac{T^2}{4\alpha}\|a\|_{\infty}\|{\rm div}\,b\|_{\infty}\bigg\}{\mathcal I}_2(w).
\end{align*}
First we fix $\varepsilon^2=a_0^2\alpha/2$ and get
\begin{align*}
\|{\mathcal L}_{\lambda}w\|_2^2\ge&
\Big\{\left [2a_0^2\alpha\rho-C_5(\|a\|_{1,\infty},\|\psi\|_{2,\infty},\alpha)-4\|a\|_{\infty}^2\right ]\lambda^3\\
&\;\,-C_4(\|a\|_\infty,\|{\rm div}\,b\|_{\infty},\alpha,T)\lambda^2
-C_1(\|a\|_{2,\infty},\|\psi\|_{3,\infty},\rho,\alpha,T)\lambda^{3/2}\\
&\;\,
-C_3(\|{\rm div}\,b\|_{\infty},\rho,\alpha,T)\lambda-\frac{a_0\sqrt{\alpha}}{\sqrt{2}}C_8(\|a\|_{1,\infty},\|\psi\|_{3,\infty},\alpha,T)\lambda\\
&\;\,-C_6(\|\psi\|_{\infty},\rho,\alpha,T)
\lambda\Big\}{\mathcal I}_1(w)\no\\
&+\bigg \{2\left [\sqrt{2\alpha}\,a_0^2\rho^{1/2}- C_2(\|a\|_{1,\infty},\|\psi\|_{2,\infty},\alpha)\right ]\lambda\\
&\quad\;\;\;\;-C_1(\|a\|_{2,\infty},\|\psi\|_{3,\infty},\rho,\alpha,T)\lambda^{1/2}
-\frac{T^2}{4\alpha}\|a\|_{\infty}\|{\rm div}\,b\|_{\infty}\bigg\}{\mathcal I}_2(w).
\end{align*}
We now choose $\rho=\rho_0$ so as to satisfy the inequalities
\begin{align*}
\left\{
\begin{array}{l}
\rho\ge 1,\\[3mm]
2a_0^2\alpha\rho-C_5(\|a\|_{1,\infty},\|\psi\|_{2,\infty},\alpha)-4\|a\|_{\infty}^2\ge 1,\\[3mm]
\displaystyle
\sqrt{2\alpha}\,a_0^2\rho^{1/2}- C_2(\|a\|_{1,\infty},\|\psi\|_{2,\infty},\alpha)\ge 1.
\end{array}
\right.
\end{align*}
Corresponding to $\rho_0$ we determine $\l_0$ such that the following inequalities are satisfied for all $\l\ge\l_0$:
\begin{align*}
\left\{
\begin{array}{l}
\lambda^3-C_4(\|a\|_{\infty},\|{\rm div}\,b\|_{\infty},\alpha,T)\lambda^2
-C_1(\|a\|_{2,\infty},\|\psi\|_{3,\infty},\rho_0,\alpha,T)\lambda^{3/2}\\[2mm]
 \displaystyle-C_3(\|{\rm div}\,b\|_{\infty},\rho_0,\alpha,T)\lambda-\frac{a_0\sqrt{\alpha}}{\sqrt{2}} C_8(\|a\|_{1,\infty},\|\psi\|_{3,\infty},\alpha,T)\lambda
\\[2mm]
 \displaystyle-C_6(\|\psi\|_{\infty},\rho_0,\alpha,T)\lambda
 \ge \frac{1}{2}\lambda^3,
 \\[4mm]
\displaystyle 2\lambda-C_1(\|a\|_{2,\infty},\|\psi\|_{3,\infty},\rho_0,\alpha,T)\lambda^{1/2}-\frac{T^2}{4\alpha}\|a\|_{\infty}
\|{\rm div}\,b\|_{\infty}\ge \frac{1}{8}\lambda,
 \\[4mm]
\displaystyle  \l \ge \frac{T^2}{\a}.
\end{array}
\right.
\end{align*}
Consequently, for all $\l\ge\l_0$ we deduce the estimate
\begin{align*}
& \int_{Q_T}\left (\frac{1}{4}\l|\nabla_x \varphi_{\rho_0}| |\nabla_xw_{\rho_0}|^2
+\l^3|\nabla_x \varphi_{\rho_0}|^3w^2_{\rho_0}\right )\rmd t\rmd x\rmd y\le 2\|{\mathcal L}_{\lambda}w_{\rho_0}\|_2^2,
\end{align*}
where, from now on, we write the dependence of $\varphi$ and $w$ on $\rho_0$.

We can now come back to our original solution $v$ using formula (\ref{def-w}).
Observe that
\begin{align*}
&\int_{Q_T}|\nabla_x\varphi_{\rho_0}||\nabla_xw_{\rho_0}|^2\,\rmd t\rmd x\rmd y\\
\ge &\int_{Q_T}|\nabla_x\varphi_{\rho_0}||\nabla_xv|^2{\rm e}^{2\l \varphi_{\rho_0}}\,\rmd t\rmd x\rmd y
+\lambda^2\int_{Q_T}|\nabla_x\varphi_{\rho_0}|^3v^2{\rm e}^{2\l \varphi_{\rho_0}}\,\rmd t\rmd x\rmd y\\
&-\int_{Q_T}(|\nabla_x\varphi_{\rho_0}|^{1/2}|\nabla_xv|{\rm e}^{\l \varphi_{\rho_0}})(2\lambda|\nabla_x\varphi_{\rho_0}|^{3/2}|v|{\rm e}^{\l \varphi_{\rho_0}})\,\rmd t\rmd x\rmd y\\
\ge &\frac{3}{4}\int_{Q_T}|\nabla_x\varphi_{\rho_0}||\nabla_xv|^2{\rm e}^{2\l \varphi_{\rho_0}}\,\rmd t\rmd x\rmd y
-3\lambda^2\int_{Q_T}|\nabla_x\varphi_{\rho_0}|^3v^2{\rm e}^{2\l \varphi_{\rho_0}}\rmd t\rmd x\rmd y,
\end{align*}
where we used the inequality $|\gamma\delta|\le \gamma^2/4+\delta^2$ which holds for any $\gamma,\delta\in\R$.
Consequently, owing to (\ref{6.6}), we get
\begin{align*}
&\int_{Q_T} \left (\frac{3}{16}\lambda |\nabla_x\varphi_{\rho_0}||\nabla_xv|^2
+ \frac{1}{4}\l^3|\nabla_x\varphi_{\rho_0}|^3v^2\right ){\rm e}^{2\l \varphi_{\rho_0}}\,\rmd t\rmd x\rmd y
\no\\[2mm]
\le & 2\int_{Q_T} |{\mathcal P}_0v|^2\,{\rm e}^{2\l \varphi_{\rho_0}}\rmd t\rmd x\rmd y.
\end{align*}
The Carleman estimate (\ref{6.72}) now follows at once.
\end{proof}

\section{A continuous dependence result for the ill-posed problem \eqref{5.1}}
\label{sect-4}
\setcounter{equation}{0}
Introduce now the family of functions $\s_\ve\in W^{1,\infty}((0,T))$, $\ve\in (0,1/2)$, defined by
\begin{equation}
\s_\ve(t) =\left\{
\begin{array}{ll}
0,\q & t\in [0,\varepsilon T],
\\[3mm]
\displaystyle\frac{t-\varepsilon T}{\varepsilon T},\q & t\in (\varepsilon T,2\varepsilon T),
\\[3mm]
1,\q & t\in [2\varepsilon T,T].
\end{array}
\right.
\label{chi}
\end{equation}
Introduce also the function
$v_\ve = \s_\ve v$, where $v$ is the solution to problem \eqref{5.4}. It is a simple task to show that $v_\ve
\in H^1((0,T);L^2_{\C}(\Om\times{\mathcal O}))\cap L^2((0,T);{\mathcal H}^2_{\C}(\Om\times{\mathcal O}))$
solves the following
initial and boundary-value problem:
\begin{align*}
\left\{
\begin{array}{l}
D_tv_\ve(t,x,y)={\rm div}_x(a(x)\nabla_xv_\ve(t,x,y))+ c(x,y)\cdot \nabla_xv_\ve(t,x,y)+\sigma_{\ve}'(t)v(x,y)\\[1mm]
\qquad\;\,\qquad\qquad
+ b(y)\cdot \nabla_yv_\ve(t,x,y)+b_0(x,y)v_\ve(t,x,y) + {\widetilde g}_\ve(t,x,y),\\[1mm]
\qquad\qquad\qquad\qquad\qquad\qquad\qquad\qquad\;\;\;\quad\qquad\qquad (t,x,y)\in Q_T,
\\[1mm]
v_\ve(t,x,y)=0, \qquad\qquad\qquad\qquad\qquad\qquad\qquad\quad\;\; (t,x,y)\in [0,T]\times\partial_*(\Omega\times{\mathcal O}),
\\[1mm]
D_{\nu}v_{\ve}(t,x,y)=0,\qquad\qquad\qquad\qquad\qquad\qquad\qquad\; (t,x,y)\in [0,T]\times\Gamma\times {\mathcal O},
\\[1mm]
v_\ve(0,x,y)=0, \qquad\qquad\qquad\qquad\qquad\qquad\qquad\quad\;\; (x,y)\in \Omega\times{\mathcal O},
\end{array}
\right.
\end{align*}
where $\widetilde g_{\ve}=\sigma_{\ve}\widetilde g$.
Multiplying the differential equation by $2\overline{v_{\ve}}$ and integrating once by parts over $\Om\times {\mathcal O}$ we obtain the identity
\begin{align*}
&\int_{{\Om\times {\mathcal O}}}D_tv_{\ve}(t,x,y)\overline{v_{\ve}(t,x,y)}\rmd x\rmd y
+\int_{{\Om\times {\mathcal O}}}a(x)|\nabla_xv_\ve(t,x,y)|^2\rmd x\rmd y
\no\\
=&\int_{{\Om\times {\mathcal O}}}\overline{v_\ve(t,x,y)}(c(x,y)\cdot \nabla_xv_\ve(t,x,y))\rmd x\rmd y\no\\
&+\int_{{\Om\times {\mathcal O}}} \overline{v_{\ve}(t,x,y)}(b(y)\cdot\nabla_yv_\ve(t,x,y))\rmd x\rmd y
\no\\
&+\int_{{\Om\times {\mathcal O}}} b_0(x)|v_\ve(t,x,y)|^2\rmd x\rmd y
+\int_{{\Om\times {\mathcal O}}}{\widetilde g}_\ve(t,x,y)\overline{v_\ve(t,x,y)}\,\rmd x\rmd y
\no\\
&+\int_{{\Om\times {\mathcal O}}}\s'_\ve(t)v(t,x,y)\overline{v_\ve(t,x,y)}\,\rmd x\rmd y,
\end{align*}
for any $t\in (0,T)$.
Taking the real part of both the sides of the previous equality and observing that
\begin{align*}
{\rm Re}\left (\int_{{\Om\times {\mathcal O}}}D_tv_{\ve}(t,x,y)\overline{v_{\ve}(t,x,y)}\rmd x\rmd y\right )=
\frac{1}{2}D_t\int_{{\Om\times {\mathcal O}}}|v_{\ve}(t,x,y)|^2\rmd x\rmd y
\end{align*}
and
\begin{align*}
&{\rm Re}\left (\int_{{\Om\times {\mathcal O}}} \overline{v_{\ve}(t,x,y)}(b(y)\cdot\nabla_yv_\ve(t,x,y))\rmd x\rmd y\right )\\
= &\int_{{\Om\times {\mathcal O}}}b(y)\cdot{\rm Re}(\overline{v_{\ve}(t,x,y)}\nabla_y v_{\ve}(t,x,y))\rmd x\rmd y\\
= & \frac{1}{2}\int_{{\Om\times {\mathcal O}}}b(y)\cdot\nabla_y|v_{\ve}(t,x,y)|^2 \rmd x\rmd y\\
=&-\frac{1}{2} \int_{{\Om\times {\mathcal O}}}({\rm div}\,b(y))|v_{\ve}(t,x,y)|^2 \rmd x\rmd y,
\end{align*}
we get
\begin{align*}
&D_t\|v_\ve(t,\cdot,\cdot)\|^2_{L^2({\Om\times {\mathcal O}})}
+ 2\int_{{\Om\times {\mathcal O}}}a(x)|\nabla_xv_\ve(t,x,y)|^2\rmd x\rmd y
\no\\
=&2\int_{{\Om\times {\mathcal O}}}\overline{v_\ve(t,x,y)}(c(x,y)\cdot \nabla_xv_\ve(t,x,y))\rmd x\rmd y\no\\
&-\int_{{\Om\times {\mathcal O}}}({\rm div}\,b(y))|v_{\ve}(t,x,y)|^2\rmd x\rmd y
\no\\
&+ 2\int_{{\Om\times {\mathcal O}}} b_0(x)|v_\ve(t,x,y)|^2\rmd x\rmd y
+ 2\int_{{\Om\times {\mathcal O}}}{\widetilde g}_\ve(t,x,y)\overline{v_\ve(t,x,y)}\,\rmd x\rmd y
\no\\
&+ 2\int_{{\Om\times {\mathcal O}}}\s'_\ve(t)v(t,x,y)\overline{v_\ve(t,x,y)}\,\rmd x\rmd y,
\end{align*}
for any $t\in (0,T)$. Therefore, using the elementary inequality
\begin{align*}
2|\overline{v_\ve} (c\cdot \nabla_xv_\ve)|\le 2\|c\|_{\infty}|v_\ve||\nabla_xv_\ve|
\le a_0^{-1}\|c\|^2_{\infty}|v_\ve|^2 + a_0|\nabla_xv_\ve|^2,
\end{align*}
$a_0$ being the positive constant in Hypothesis \ref{hyp-1}(i), we can estimate
\begin{align}
&D_t\|v_\ve(t,\cdot,\cdot)\|^2_{L^2_{\C}({\Om\times {\mathcal O}})}
+a_0\|\nabla_xv_\ve(t,\cdot,\cdot)\|_{L^2_{\C}(\Omega\times{\mathcal O})}^2
\no\\
\le &\left (\|{\rm div}\,b\|_{\infty}+a_0^{-1}\|c\|_{\infty}^2+ 2\|b_0\|_{\infty}\right )\|v_\ve(t,\cdot,\cdot)\|_{L^2_{\C}(\Omega\times{\mathcal O})}^2\no\\
&+ 2\|{\widetilde g}_\ve(t,\cdot,\cdot)\|_{L^2_{\C}(\Omega\times{\mathcal O})}\|v_\ve(t,\cdot,\cdot)\|_{L^2_{\C}(\Omega\times\mathcal{O})}+ 2|\s'_\ve(t)|\|v(t,\cdot,\cdot)\|_{L^2_{\C}(\Omega\times {\mathcal O})}^2.
\label{cont-dep}
\end{align}

We now fix $\tau\in (0,T]$ and integrate \eqref{cont-dep} with respect to $t$ over $(0,\tau)$. Taking \eqref{chi} into account, we obtain
\begin{align}
z_\ve(\tau):=&\|v_\ve(\tau,\cdot,\cdot)\|^2_{L^2_{\C}({\Om\times {\mathcal O}})}
+ a_0\|\nabla_xv_\ve\|^2_{L^2_{\C}(Q_\tau)}\no\\
\le &(\|{\rm div}\,b\|_{\infty} + a_0^{-1}\|c\|^2_{\infty} + 2\|b_0\|_{\infty})
\int_0^\tau \|v_\ve(t,\cdot,\cdot)\|^2_{L^2_{\C}({\Om\times {\mathcal O}})}\,\rmd t
\no\\
&+ 2\int_0^\tau \|{\widetilde g}_\ve(t,\cdot,\cdot)\|_{L^2_{\C}({\Om\times {\mathcal O}})}
\|v_\ve(t,\cdot,\cdot)\|_{L^2_{\C}({\Om\times {\mathcal O}})}\,\rmd t\no\\
&+ \frac{2}{\ve T}\int_{\ve T}^{2\ve T} \|v(t,\cdot,\cdot)\|^2_{L^2_{\C}({\Om\times {\mathcal O}})}\,\rmd t.
\label{7.5}
\end{align}
The Carleman estimate \eqref{6.72} yields the inequality
\begin{equation}
\label{7.7}
\int_{\ve T}^{2\ve T} \|v(t,\cdot,\cdot)\|^2_{L^2_{\C}({\Om\times {\mathcal O}})}\,\rmd t
\le M_1(\ve,T)\|{\widetilde g}\|^2_{L^2_{\C}(Q_T)},
\end{equation}
where the constant $M_1(\ve,T)$ depends also on $a_0$, $\|a\|_{2,\infty}$, $\|b_0\|_{\infty}$, $\|{\rm div}\,b\|_{\infty}$,
$\|c\|_{\infty}$, $\|\psi\|_{3,\infty}$ and $\alpha$.

From \eqref{7.5} and \eqref{7.7} we obtain the following integral inequality for function $z_\ve$:
\begin{align}
z_\ve(\tau) \le \b \int_0^\tau z_\ve(t)\,\rmd t
+ 2\int_0^\tau \|{\widetilde g}(t,\cdot,\cdot)\|_{L^2_{\C}({\Om\times {\mathcal O}})}z_\ve(t)^{1/2}\,\rmd t
+M_2(\ve,T)\|{\widetilde g}\|^2_{L^2_{\C}(Q_T)},
\label{7.8}
\end{align}
for any $\tau\in (0,T]$, where
$\b = \|{\rm div}\,b\|_{\infty}+ 2a_0^{-1}\|c\|^2_{\infty}+ 2\|b_0\|_{\infty}$.
Then, we need \cite[Theorem 4.9]{BS}, with $p=1/2$, which we report here as a lemma.

\begin{lemma} Let $z:[0,T]\to\R$ be a nonnegative continuous function and let $b,k\in L^1((0,T))$ be nonnegative functions satisfying
\begin{align*}
z(t) \le \gamma + \int_0^t b(s)z(s)\,\rmd s + \int_0^t k(s)z(s)^p\,\rmd s,\qq t\in [0,T],
\end{align*}
where $p\in (0,1)$ and $\gamma\ge 0$ are given constants. Then, for all $t\in [0,T]$
\begin{align*}
z(t) \le\exp \left( \int_0^t b(s)\,\rmd s\right) \left[\gamma^{1-p}
+ (1-p)\int_0^t k(s)\exp \left( (p-1)\int _{0}^{s}b(\s)\,d\s \right) \,\rmd s\right] ^{\frac{1}{1-p}}.
\end{align*}
\end{lemma}

From this lemma and \eqref{7.8} we deduce the fundamental estimate holding true for all $\tau \in [0,T]$:
\begin{align*}
z_{\varepsilon}(\tau)=&\|v_\ve(\tau,\cdot,\cdot)\|^2_{L^2_{\C}({\Om\times {\mathcal O}})}
+ a_0\|\nabla_xv_\ve\|^2_{L^2_{\C}(Q_\tau)}
\no\\
\le& \left (M_2(\ve,T)^{1/2}\|{\widetilde g}\|_{L^2_{\C}(Q_T)}{\rm e}^{\b \tau/2}
+ \int_0^\tau {\rm e}^{\b(\tau-s)/2}\|{\widetilde g}_\ve(s,\cdot,\cdot)\|_{L^2_{\C}({\Om\times {\mathcal O}})}\,\rmd s\right )^2\no\\
\le &2M_2(\ve,T)\|{\widetilde g}\|^2_{L^2_{\C}(Q_T)}{\rm e}^{\b \tau}
+ 2\left (\int_0^\tau {\rm e}^{\b(\tau-s)/2}\|{\widetilde g}(s,\cdot,\cdot)\|_{L^2_{\C}({\Om\times {\mathcal O}})}\,\rmd s\right )^2
\no\\
\le &2\big(M_2(\ve,T)+\b^{-1}\big )\|{\widetilde g}\|^2_{L^2_{\C}(Q_T)}{\rm e}^{\b \tau},
\end{align*}
for any $\tau\in [0,T]$.
In particular, for all $\tau\in [2\ve T,T]$ we find the following estimate for $v$, where we have set
$Q(2\ve T,\tau)=(2\ve T,\tau)\times {\Om\times {\mathcal O}}$:
\begin{align}
\|v(\tau,\cdot,\cdot)\|^2_{L^2_{\C}({\Om\times {\mathcal O}})}
+ a_0\|\nabla_xv\|^2_{L^2_{\C}(Q(2\ve T,\tau))}
\le 2\big[M_2(\ve,T)+\b^{-1}\big]\|{\widetilde g}\|^2_{L^2_{\C}(Q_T)}{\rm e}^{\b \tau},
\label{7.11}
\end{align}
for any $\tau\in [0,T]$.
Recalling that the solution $u$ to problem \eqref{5.1} is related to $v$ by the formula $u=v+h$, from \eqref{7.11} we immediately deduce the estimate for $u$:
\begin{align}
&\|u(\tau,\cdot,\cdot)\|^2_{L^2_{\C}({\Om\times {\mathcal O}})}
+ a_0\|\nabla_xu\|^2_{L^2_{\C}(Q(2\ve T,\tau))}\no\\
\le & 2\|h(\tau,\cdot,\cdot)\|^2_{L^2_{\C}({\Om\times {\mathcal O}})}
+ 2a_0\|\nabla_xh\|^2_{L^2_{\C}(Q(2\ve T,\tau))}
+ 4\big[M_2(\ve,T)+\b^{-1}\big]\|{\widetilde g}\|^2_{L^2_{\C}(Q_T)}{\rm e}^{\b \tau},
\label{7.12}
\end{align}
for any $\tau\in [2\ve T,T]$.
Now, taking advantage of definition \eqref{5.5}, we can estimate
\begin{align}
& \|{\widetilde g}\|^2_{L^2_{\C}(Q_T)}\le 6\|g\|^2_{L^2_{\C}(Q_T)}+6\|D_th\|^2_{L^2_{\C}(Q_T)}
+6\|{\rm div}_x(a\nabla_xh)\|_{L^2_{\C}(Q_T)}^2
\no\\[1mm]
& \phantom{\|{\widetilde g}\|^2_{L^2_{\C}(Q_T)}\le}
 + 6\|c\|^2_\infty\|\nabla_xh\|^2_{L^2_{\C}(Q_T)}+6\|b\cdot\nabla_yh\|^2_{L^2_{\C}(Q_T)}
+6\|b_0\|^2_\infty\|h\|^2_{L^2_{\C}(Q_T)}.
\label{7.13}
\end{align}
Finally, \eqref{7.12} and \eqref{7.13} yield the following continuous dependence estimate
for all $\tau\in [2\ve T,T]$:
\begin{align}
& \|u(\tau,\cdot,\cdot)\|^2_{L^2_{\C}(\Om\times {\mathcal O})}
+ a_0\|\nabla_xu\|^2_{L^2_{\C}(Q(2\ve T,\tau))}\no\\
\le &M_3(\ve,T)\Big\{\|g\|^2_{L^2_{\C}(Q_T)}
+\|h\|^2_{H^1((0,T);L^2_{\C}(\Om\times {\mathcal O}))}
+\|{\rm div}_x(a\nabla_xh)\|_{L^2_{\C}(Q_T)}^2\no\\
&\qquad\qquad\;\;+\|b\cdot\nabla_yh\|^2_{L^2_{\C}(Q_T)} + \|\nabla_xh\|^2_{L^2_{\C}(Q_T)}\Big\},
\label{7.14}
\end{align}
where the positive constant $M_3$ depends also on $a_0$, $\|a\|_{2,\infty}$, $\|b_0\|_{\infty}$, $\|{\rm div}\,b\|_{\infty}$,
$\|c\|_\infty$, $\|\psi\|_{3,\infty}$ and $\alpha$.

We have so proved the following continuous dependence result:

\begin{theorem}
\label{thm-4.9}
Under Hypotheses $\ref{hyp-1}$ the solution $u$ to problem $\eqref{5.1}$ satisfies the continuous dependence estimate $\eqref{7.14}$.
\end{theorem}

\section{Some extensions of our main results}
\label{sect-5}

In this section we show that the validity of Theorem \ref{thm-4.9} can be extended both to some classes of
degenerate integrodifferential boundary problems and to some classes of semilinear problems.

\subsection{A degenerate convolution integrodifferential problem}
\label{subsect-5.1}
\setcounter{equation}{0}
Here we consider a
convolution integrodifferential problem with no initial conditions, and with Cauchy data on the lateral boundary of the cylinder $\Omega\times \R^N$. We still assume that $\Omega$ is a bounded subset of $\R^M$ with a boundary of class $C^3$.

Let $\widetilde {\mathcal A}$ be the following degenerate integrodifferential linear operator
\begin{align*}
\widetilde {\mathcal A}z(x,y) =&  {\rm div}_x(a(x)\nabla_xz(x,y))
+By\cdot \nabla_yz(x,y) + \b_1(x)\cdot \nabla_xz(x,y)
\no\\
& +\b_0(x)z(x,y) + \int_{\R^N} k_1(x,y-\eta)\cdot \nabla_xz(x,\eta)\,d\eta\no\\
&+ \int_{\R^N} k_0(x,y-\eta)z(x,\eta)\,d\eta.
\end{align*}
Consider the parabolic integrodifferential problem with no initial condition, but with Cauchy data on the boundary
\begin{equation}
\left\{
\begin{array}{ll}
D_tz(t,x,y)=\widetilde {\mathcal A}z(t,x,y)+f(t,x,y),\q  & (t,x,y)\in (0,T)\times \Om\times \R^N,\\[1mm]
z(t,x,y)=h(t,x,y), & (t,x,y)\in [0,T]\times\partial\Omega\times \R^N,\\[1mm]
D_{\nu}z(t,x,y)=D_{\nu}h(t,x,y),& (t,x,y)\in [0,T]\times\Gamma\times\R^N,
\end{array}
\right.
\label{8.3}
\end{equation}
where $\Gamma$ is an open subset of $\partial\Omega$ and
\begin{hyp}
\label{hyp-2}
\begin{enumerate}[\rm (i)]
The following conditions are satisfied:
\item
$a\in W^{2,\infty}(\Om)$ and there exists a positive constant $a_0$ such that $|a(x)|\ge a_0$ for any $x\in\Omega$;
\item
$B$ is a real $(N\times N)$-square matrix;
\item
$\b_0\in L^{\infty}(\Om)$;
\item
$\b_1\in (L^{\infty}(\Om))^M$;
\item
$k_0\in L^\infty(\Om;L^1(\R^N))$;
\item
$k_1\in L^\infty(\Om;L^1(\R^N))^M$;
\item
$f\in L^2(Q_T)$;
\item
$h\in H^1((0,T);L^2(\Om\times {\mathcal O}))\cap L^2((0,T);{\mathcal H}^2(\Om\times {\mathcal O}))$.
\end{enumerate}
\end{hyp}

Denote by ${\mathcal F}_y$ the Fourier transform with respect to the variable $y$.
As it is easily seen, function $u={\mathcal F}_yz$ solves the ill-posed problem
\begin{align*}
\left\{
\begin{array}{l}
D_tu(t,x,\eta)- {\rm div}_x(a(x)\nabla_xu(t,x,\eta))
+c(x,\eta)\cdot \nabla_xu(t,x,\eta)\\[1mm]
+ B^T\eta\cdot \nabla_\eta u(t,x,\eta)+b_0(x,\eta)u(t,x,\eta)=({\mathcal F}_yf)(t,x,\eta),\\[1mm]
\qquad\qquad\qquad\qquad\qquad\qquad\qquad\qquad\qquad\qquad\qquad\;\;\;\; (t,x,\eta)\in (0,T)\times \Om\times \R^N,\\[1mm]
u(t,x,\eta)=({\mathcal F}_yh)(t,x,\eta),\qquad\qquad\qquad\qquad\qquad\quad\quad (t,x,\eta)\in [0,T]\times\partial\Omega\times \R^N,\\[1mm]
D_{\nu}u(t,x,\eta)=({\mathcal F}_yD_{\nu}h)(t,x,\eta),\qquad\qquad\qquad\qquad\quad\;\;\! (t,x,\eta)\in [0,T]\times\Gamma\times\R^N,
\end{array}
\right.
\end{align*}
where $c:\Omega\times \R^N\to {\mathbb C}^M$ and $\b_0:\Omega\times \R^N\to {\mathbb C}$ are defined by
\begin{align*}
c(x,\eta)=-\b_1(x)-({\mathcal F}_yk_1)(x,\eta),\qquad\;\, b_0(x,\eta)={\rm Tr}(B)-\b_0(x)-({\mathcal F}_yk_0)(x,\eta),
\end{align*}
for any $x\in\Omega$ and $\eta\in\R^N$. By Theorem \ref{thm-4.9}, $u$ satisfies the continuous dependence estimate
\begin{align}
&\|u(\tau,\cdot,\cdot)\|^2_{L^2(\Om\times {\mathcal O})}
+ a_0\|\nabla_xu\|^2_{L^2(Q(2\ve T,\tau))}\no\\
\le &M(\ve,T)\Big\{\|{\mathcal F}_yf\|^2_{L^2(Q_T)}
+\|{\mathcal F}_yh\|^2_{H^1((0,T);L^2(\Om\times {\mathcal O}))}
+\|{\rm div}_x(a\nabla_x{\mathcal F}_yh)\|_{L^2(Q_T)}^2\no\\
&\qquad\qquad\;\;+\|B^T\eta\cdot\nabla_{\eta}{\mathcal F}_yh\|^2_{L^2(Q_T)} + \|\nabla_x{\mathcal F}_yh\|^2_{L^2(Q_T)}\Big\},
\label{7.14-bis}
\end{align}
for all $\varepsilon\in (0,1/4)$,
$\tau\in [2\ve T,T]$ and
some positive constant $M(\varepsilon,T)$, depending also on $a_0$, $\|a\|_{2,\infty}$, $\|B\|$, $\|\beta_0\|_{\infty}$, $\|\beta_1\|_{\infty}$, $\|\psi\|_{3,\infty}$, $\|k_0\|_{L^{\infty}(\Omega;L^1(\R^N))}$, $\|k_1\|_{L^{\infty}(\Omega;L^1(\R^N))^M}$ and $\a$.

Using the Parseval identity and observing that $\nabla_x$ commutes with ${\mathcal F}_y$ and that $\varphi_{\rho_0}$ is {\it independent} of $\eta$,
from \eqref{7.14-bis} we deduce that $z$ satisfies
\begin{align}
&\|z(\tau,\cdot,\cdot)\|^2_{L^2({\Om\times {\mathcal O}})}
+ a_0\|\nabla_xz\|^2_{L^2(Q(2\ve T,\tau))}\no\\
\le & M(\ve,T)\Big\{\|f\|^2_{L^2(Q_T)}
+\|h\|^2_{H^1((0,T);L^2(\Om\times {\mathcal O}))}
+\|{\rm div}_x(a\nabla_xh)\|_{L^2(Q_T)}^2\no\\
&\qquad\qquad\;\,
+\|By\cdot\nabla_yh\|^2_{L^2(Q_T)} + \|\nabla_xh\|^2_{L^2(Q_T)}\Big\},
\label{8.15}
\end{align}
for all $\varepsilon\in (0,1/4)$ and $\tau\in [2\ve T,T]$.

We have proved the following continuous dependence result:

\begin{theorem}
Let Hypotheses $\ref{hyp-2}$ be satisfied. Then,
the solution $z$ to problem $\eqref{8.3}$ satisfies the continuous dependence estimate $\eqref{8.15}$. In particular, if $(\b_0,\b_1,f,h)=(0,0,0,0)$, then $z=0$
in $Q_T$, i.e., the unique continuation property holds true.
\end{theorem}

\subsection{A semilinear parabolic equation}
\label{subsect-5.2}

We now consider the following semilinear boundary value problem
\begin{equation}
\left\{
\begin{array}{l}
D_tu(t,x,y)={\rm div}_x(a(x)\nabla_xu(t,x,y)) + b(y)\cdot\nabla_yu(t,x,y)
\\[1mm]
\qquad\qquad\qquad\, + q(u(t,x,y),\nabla_xu(t,x,y)) + g(t,x,y),\\[1mm]
\qquad\qquad\qquad\qquad\qquad\qquad\qquad\qquad\qquad\;\;\; (t,x,y)\in [0,T]\times\Omega\times{\mathcal O}=:Q_T,
\\[1mm]
u(t,x,y)=h(t,x,y),\qquad\qquad\qquad\qquad\quad\;\;\, (t,x,y)\in [0,T]\times\partial_*(\Omega\times{\mathcal O}),
\\[1mm]
D_{\nu}u(t,x,y)=D_{\nu}h(t,x,y),\qquad\qquad\qquad\quad\;\! (t,x,y)\in [0,T]\times\Gamma\times{\mathcal O}.
\end{array}
\right.
\label{9.1}
\end{equation}
where $\Omega$ and ${\mathcal O}$ $a$ and $b$ are as in the previous sections (see Hypothesis \ref{hyp-1}), whereas function $q$ satisfies the following condition
\begin{hyp}
$q:\C^{n+1}\to\C$ is a Lipschitz-continuous function with a Lipschitz constant $\kappa$.
\end{hyp}

\begin{theorem} Let $u_j\in H^1((0,T);L^2_{\C}(\Om\times {\mathcal O}))\cap L^2((0,T);{\mathcal H}^2_{\C}(\Om\times {\mathcal O}))$ be a solution to problem $\eqref{9.1}$ corresponding to $(g,h)=(g_j,h_j)$, $j=1,2$. Then, for any $\varepsilon\in (0,1/2)$, there exists a positive constant $C=C(\varepsilon,T)$ such that
\begin{align}
&\|u_2(\tau,\cdot,\cdot)-u_1(\tau,\cdot,\cdot)\|^2_{L^2_{\C}(\Om\times {\mathcal O})}
+ a_0\|\nabla_xu_2-\nabla_xu_1\|^2_{L^2_{\C}(Q(2\ve T,\tau))}\no\\
\le &C\Big\{\|g_2-g_1\|^2_{L^2_{\C}(Q_T)}
+\|h_2-h_1\|^2_{H^1((0,T);L^2_{\C}(\Om\times {\mathcal O}))}
+\|{\rm div}_x(a\nabla_xh_2-a\nabla_xh_1)\|_{L^2_{\C}(Q_T)}^2\no\\
&\quad\;\,+\|b\cdot\nabla_yh_2-b\cdot\nabla_yh_1\|^2_{L^2_{\C}(Q_T)} + \|\nabla_xh_2-\nabla_xh_1\|^2_{L^2_{\C}(Q_T)}\Big\},
\label{9.11}
\end{align}
for any $\tau\in [2\varepsilon T,T]$.
\end{theorem}

\begin{proof}
The proof follows adapting the arguments in Sections \ref{sect-3} and \ref{sect-4}. Hence, we just point out the differences.

Note that, if we set
\begin{align*}
{\mathcal P}(u)=D_tu-{\rm div}_x(a\nabla_xu)-b\cdot\nabla_yu-q(u,\nabla_xu)=:{\mathcal P}_0(u)-q(u,\nabla_xu),
\end{align*}
we can rewrite the differential equation in \eqref{9.1} in the much more compact form:
${\mathcal P}u=g$.

First we perform the translations $v_j=u_j-h_j$, $j=1,2$, and observe that the function $v=v_2-v_1$ solves problem \eqref{9.1} with $q(v,\nabla_xv)$ and $(g,h)$ being replaced,
respectively, by $Q(v_1,v_2)$ and $(g_2-g_1,0)$, where
\begin{align}
\label{9.3}
Q(v_1,v_2)=q(v_2+h_2,\nabla_xv_2+\nabla_x h_2)-q(v_1+h_1,\nabla_xv_1+\nabla_x h_1)
\end{align}
and
\begin{equation}
\label{9.2}
{\widetilde g}_j=g-D_th_j+{\rm div}_x(a\nabla_xh_j)+b\cdot\nabla_yh_j,\qquad\;\, j=1,2.
\end{equation}
Moreover,
\begin{equation}
{\mathcal P}(v)={\mathcal P}_0(v)-Q(v_1,v_2).
\label{vv}
\end{equation}
Since $q$ is a Lipschitz continuous function in $\C^{n+1}$, we can estimate (pointwise)
\begin{equation}
|Q(v_1,v_2)|\le  h_q\big (|v_2-v_1|+|\nabla_x(v_2-v_1)|\big )
+h_q\big (|h_2-h_1|+|\nabla_x(h_2-h_1)|\big ).
\label{vvv}
\end{equation}
From the Carleman estimate in Theorem \ref{thm-main2}, \eqref{vv} and \eqref{vvv}, we can infer that $v$ satisfies the following integral inequality for all $\l\ge \l_0$:
\begin{align*}
\phantom{\le\ }
&\int_{Q_T} \big (\l|\nabla_x\varphi_{\rho_0}||\nabla_xv|^2 + \l^3|\nabla_x\varphi_{\rho_0}|^3|v|^2\big)
{\rm e}^{2\l \varphi_{\rho_0}}\,\rmd t\rmd x\rmd y\\
\le &\frac{32}{3}\int_{Q_T} |{\mathcal P}_0v|^2
{\rm e}^{2\l \varphi_{\rho_0}}\,\rmd t\rmd x\rmd y
\nonumber\\[1mm]
\le &\frac{160}{3}\int_{Q_T} |{\mathcal P}(v_2)-{\mathcal P}(v_1)|^2{\rm e}^{2\l \varphi_{\rho_0}}\,\rmd t\rmd x\rmd y
\nonumber\\[1mm]
\phantom{\le}
& + \frac{160}{3}\kappa^2\int_{Q_T} \big[|v_2-v_1|^2+|\nabla_x(v_2-v_1)|^2\big]
{\rm e}^{2\l \varphi_{\rho_0}} \,\rmd t\rmd x\rmd y
\nonumber\\[1mm]
\phantom{\le}
&+ \frac{160}{3}\kappa^2\int_{Q_T} \big[|(h_2-h_1)|^2+|\nabla_x(h_2-h_1)|^2\big]
{\rm e}^{2\l \varphi_{\rho_0}} \,\rmd t\rmd x\rmd y
\nonumber\\[1mm]
\le &\frac{160}{3}\int_{Q_T} |{\mathcal P}(v_2)-{\mathcal P}(v_1)|^2{\rm e}^{2\l \varphi_{\rho_0}} \,\rmd t\rmd x\rmd y
\nonumber\\[1mm]
\phantom{\le}
&+ \frac{160}{3}\kappa^2\int_{Q_T} \bigg\{\bigg[ \bigg (\frac{\ell}{\alpha\rho_0}\bigg )^3|\nabla_x\varphi_{\rho_0}|^3|(v_2-v_1)(t,x,y)|^2
\nonumber\\[1mm]
\phantom{\le}
&\phantom{\le\frac{160}{3}\kappa^2\int_{Q_T} \bigg\{\bigg[\;\;}+\frac{\ell}{\alpha\rho_0}|\nabla_x\varphi_{\rho_0}||\nabla_x(v_2-v_1)|^2\bigg ]{\rm e}^{2\l \varphi_{\rho_0}}
\bigg\} \,\rmd t\rmd x\rmd y
\nonumber\\[1mm]
\phantom{\le}
&+ \frac{160}{3}\kappa^2\int_{Q_T} \big[|(h_2-h_1)|^2+|\nabla_x(h_2-h_1)|^2\big]
{\rm e}^{2\l \varphi_{\rho_0}} \,\rmd t\rmd x\rmd y.
\end{align*}
Whence we deduce the Carleman estimate for $v$:
\begin{align*}
& \phantom{\le\le}
\int_{Q_T} \big (\l|\nabla_x\varphi_{\rho_0}||\nabla_x(v_2-v_1)|^2 + \l^3|\nabla_x\varphi_{\rho_0}|^3|v_2-v_1|^2\big)
{\rm e}^{2\l \varphi_{\rho_0}}\,\rmd t\rmd x\rmd y
\nonumber\\[1mm]
&\le \frac{320}{3}\int_{Q_T} |{\mathcal P}(v_2)-{\mathcal P}(v_1)|^2{\rm e}^{2\l \varphi_{\rho_0}}
 \,\rmd t\rmd x\rmd y,
\nonumber\\[1mm]
&\phantom{\le\;\;}
+ \frac{320}{3}\kappa^2\int_{Q_T} \big[|h_2-h_1|^2+|\nabla_x(h_2-h_1)|^2\big]
{\rm e}^{2\l \varphi_{\rho_0}} \,\rmd t\rmd x\rmd y,
\end{align*}
if we choose
\begin{align*}
\l\ge\max\left\{\l_0,\bigg (\frac{320}{3}\kappa^2\bigg )^{1/3}\frac{\ell}{\alpha\rho_0},\frac{320}{3}\kappa^2\frac{\ell}{\alpha\rho_0}\right\}.
\end{align*}

Now, we are almost done. Indeed, arguing as in the proof of \eqref{cont-dep} we can show that
\begin{align*}
&D_t\|v_\ve(t,\cdot,\cdot)\|^2_{L^2_{\C}({\Om\times {\mathcal O}})}
+a_0\|\nabla_xv_\ve(t,\cdot,\cdot)\|_{L^2_{\C}(\Omega\times{\mathcal O})}^2
\no\\
\le &\|{\rm div}\,b\|_{\infty}\|v_\ve(t,\cdot,\cdot)\|_{L^2_{\C}(\Omega\times{\mathcal O})}^2+a_0^{-1}\int_{\Omega\times{\mathcal O}}\sigma_{\varepsilon}(t)|Q(v_1,v_2)||v_{\varepsilon}|\rmd x\rmd y\no\\
&+ 2\|{\widetilde g}_{2,\ve}(t,\cdot,\cdot)-{\widetilde g}_{1,\ve}(t,\cdot,\cdot)\|_{L^2_{\C}(\Omega\times{\mathcal O})}\|v_{\ve}(t,\cdot,\cdot)\|_{L^2_{\C}(\Omega\times\mathcal{O})}\nonumber\\
&+ 2|\s'_\ve(t)|\|v(t,\cdot,\cdot)\|_{L^2_{\C}(\Omega\times {\mathcal O})}^2,
\end{align*}
for any $t\in (0,T)$,
where $g_{j,\ve}=\sigma_{\varepsilon}g_j$ ($j=1,2$) and $\sigma_{\varepsilon}$ is given by
\eqref{chi}.
Since $q$ is Lipschitz continuous, we can estimate
\begin{align}
\label{9.10}
|v_{\varepsilon}||\sigma_{\varepsilon}Q(v_1,v_2)|
&\le \kappa|v_{\varepsilon}|^2+\kappa|v_{\varepsilon}||\nabla_xv_{\varepsilon}|
+\kappa|v_{\varepsilon}||h_2-h_1|+\kappa|v_{\varepsilon}|\nabla_xh_2-\nabla_xh_1|\no\\
&\le 2\kappa|v_{\varepsilon}|^2+\kappa|\nabla_xv_{\varepsilon}|^2
+\kappa|h_2-h_1|^2+\kappa|\nabla_xh_2-\nabla_xh_1|^2
\end{align}
and, consequently,
\begin{align*}
&D_t\|v_\ve(t,\cdot,\cdot)\|^2_{L^2_{\C}({\Om\times {\mathcal O}})}
+a_0\|\nabla_xv_\ve(t,\cdot,\cdot)\|_{L^2_{\C}(\Omega\times{\mathcal O})}^2
\no\\
\le &\|{\rm div}\,b\|_{\infty}\|v_\ve(t,\cdot,\cdot)\|_{L^2_{\C}(\Omega\times{\mathcal O})}^2+2a_0^{-1}\kappa\|v_{\varepsilon}(t,\cdot,\cdot)\|^2_{L^2_{\C}(\Omega\times{\mathcal O})}
\nonumber\\
&+a_0^{-1}\kappa\|\nabla_xv(t,\cdot,\cdot)\|_{L^2_{\C}(\Omega\times{\mathcal O})}+\frac{1}{2}a_0^{-1}\kappa\|h_2(t,\cdot,\cdot)-h_1(t,\cdot,\cdot)\|_{L^2_{\C}(\Omega\times{\mathcal O})}
\no\\
&+\frac{1}{2}a_0^{-1}\kappa\|\nabla_xh_2(t,\cdot,\cdot)-\nabla_xh_1(t,\cdot,\cdot)\|_{L^2_{\C}(\Omega\times{\mathcal O})}\nonumber\\
&+ 2\|{\widetilde g}_{2,\ve}(t,\cdot,\cdot)-{\widetilde g}_{1,\ve}(t,\cdot,\cdot)\|_{L^2_{\C}(\Omega\times{\mathcal O})}\|v_{\ve}(t,\cdot,\cdot)\|_{L^2_{\C}(\Omega\times\mathcal{O})}\nonumber\\
&+ 2|\s'_\ve(t)|\|v(t,\cdot,\cdot)\|_{L^2_{\C}(\Omega\times {\mathcal O})}^2.\nonumber\\
&
\end{align*}

From \eqref{9.10}, reasoning as in the previous section, we easily deduce the desired continuity estimates \eqref{7.11} for $v$, where ${\widetilde g}={\widetilde g}_2-{\widetilde g}_1$.

Now, the proof follows the same lines as the proof of Theorem \ref{thm-4.9} and yields \eqref{9.11}.
\end{proof}

\end{document}